\documentclass[11pt,a4paper,reqno]{article}

\usepackage[a4paper,width=160mm,top=25mm,bottom=30mm]{geometry}
\usepackage{hyperref, enumitem}
\usepackage[utf8]{inputenc}
\usepackage[T1]{fontenc}
\usepackage[english]{babel}
\usepackage{xcolor}
\usepackage{stmaryrd, xfrac}
\usepackage{mathrsfs} 
\usepackage{amsthm}
\usepackage{bbm}
\usepackage{enumitem}

\usepackage{tikz-cd}

\reversemarginpar
\theoremstyle{plain}
\newtheorem{theorem}{Theorem}[section]
\newtheorem{corollary}[theorem]{Corollary}
\newtheorem{proposition}[theorem]{Proposition}

\newtheorem{lemma}[theorem]{Lemma}
\newtheorem{claim}[theorem]{Claim}
\newtheorem{question}[theorem]{Question}

\theoremstyle{definition}
\newtheorem{definition}[theorem]{Definition}

\theoremstyle{remark} 
\newtheorem{remark}[theorem]{Remark}

\newcommand{\defini}{\textit}
\newcommand{\scr}{\mathscr}

\usepackage{amsmath}
\usepackage{amssymb}
\usepackage{mathtools}
\usepackage{centernot}
\mathtoolsset{showonlyrefs}

\renewcommand{\geq}{\geqslant}
\renewcommand{\leq}{\leqslant} 
\renewcommand{\ge}{\geqslant}
\renewcommand{\le}{\leqslant} 

\usepackage{calc} 

\newlength{\arrow}
\usepackage{xargs}

\newcommand{\pei}{\mathbf{C}}
\newcommand{\weak}{\mathbf{C}^\star}

\newcommand{\Ab}{\mathfrak{A}}

\newcommand{\smx}[1]{}

\definecolor{mycolor}{RGB}{81, 163, 193}

\usepackage{tabularx}
\usepackage{graphicx}
\usepackage{caption}
\usepackage{subcaption}
\begin{document}
	\title{Percolation on graphs of polynomial growth is local:\\analyticity, supercritical sharpness, isoperimetry}
	\date{\today} 
	
	\author{Sébastien Martineau\footnote{Sorbonne Universit\'e, LPSM, \nolinkurl{smartineau@lpsm.paris}.} \and Christoforos Panagiotis\footnote{University of Bath, \nolinkurl{cp2324@bath.ac.uk}}}

	\maketitle
	\begin{abstract}
    We investigate locality of the supercritical regime for Bernoulli percolation on transitive graphs with polynomial growth, by which we mean the following. Take a transitive graph of polynomial growth $\mathscr{G}$ satisfying $p_c(\mathscr{G})<1$ and take $p>p_c(\mathscr{G})$. Let $\mathscr{H}$ be another such graph and assume that $\mathscr{G}$ and $\mathscr{H}$ have the same ball of radius $r$ for $r$ large. We prove that various quantities regarding percolation of parameter close to $p$ on $\mathscr{H}$ can be well understood from $(\mathscr{G},p)$ alone.
    
    This includes uniform versions of supercritical sharpness as well as the Kesten--Zhang bound on the probability of observing a large finite cluster and the exponential cluster repulsion bound of Pete: the constants involved can be chosen to depend only on $(\mathscr{G},p)$. We also prove that $\theta_\mathscr{H}$ is an analytic function of $p$ in the whole supercritical regime and that, for a suitable $\varepsilon=\varepsilon(\mathscr{G},p)>0$, the analytic extension of $\theta_\mathscr{H}$ to the $\varepsilon$-neighbourhood of $p$ in $\mathbb C$ is, uniformly, well approximated by the analytic extension of $\theta_\mathscr{G}$. The proof relies on new results on the connectivity of minimal cutsets; in particular, we answer a question asked by Babson and Benjamini in 1999. We further discuss connections with the conjecture of non-percolation at criticality.
	\end{abstract}

\section{Introduction}

One of the simplest models of statistical mechanics is Bernoulli percolation: if one picks a random subgraph by selecting edges in an i.i.d.~way, then investigating the properties of the connected components of this random graph leads to an interesting phase transition.

More precisely, let $\mathscr{G}$ be a graph. Throughout this paper, all graphs are assumed to be simple, non-empty, locally finite, and connected. We are interested in Bernoulli bond percolation: each edge of $\mathscr{G}$ is either deleted (called closed) or retained (called open) independently with retention probability $p\in [0,1]$ to obtain a random subgraph $\omega$ of $\mathscr{G}$. We write $\mathbb{P}_{p,\mathscr{G}}$ or simply $\mathbb{P}_p$ for the law of $\omega$. Connected components of $\omega$ are referred to as \defini{clusters}. Percolation theory is primarily concerned with the structure of clusters, and how this structure depends on $p$ and on the geometry of $\mathscr{G}$. Of particular interest are {\em phase transitions}, where a small change in $p$ imposes a dramatic change of the structure.

Given a transitive graph $\mathscr{G}$ and a vertex $o$ in $\mathscr{G}$, one can introduce the function \[\theta_{\mathscr{G}}:p\longmapsto \mathbb{P}_{p,\mathscr{G}}(o\longleftrightarrow \infty),\]where $\{o\longleftrightarrow \infty\}$ denotes the event that the cluster of $o$ is infinite. The origin $o$ can be any vertex of $\mathscr{G}$ and all choices of $o$ give rise to the same function, by transitivity of $\mathscr{G}$. The function $\theta_{\mathscr{G}}$ being non-decreasing, one can define $p_c(\mathcal{G})$ to be the unique parameter in $[0,1]$ such that the following two conditions hold:
\begin{itemize}
	\item for all $p\in [0,p_c(\mathscr{G}))$, we have $\theta_\mathscr{G}(p)=0$,
	\item for all $p\in (p_c(\mathscr{G}),1]$, we have $\theta_\mathscr{G}(p)>0$.
\end{itemize}

In the most classical framework, one takes $\mathscr{G}$ to be a fixed graph within some suitable class of graphs and studies the behaviour of $\theta$ and other macroscopic observables as the value of $p$ varies.  Benjamini and Schramm \cite{MR1423907} initiated the systematic study of percolation on transitive graphs, and in recent decades, there has been growing interest in understanding how various percolation observables depend on the underlying graph $\mathscr{G}$. A major conjecture of this field was Schramm's Locality Conjecture, which we briefly review below.

Given two transitive graphs $\mathscr{G}$ and $\mathscr{H}$, introduce the following quantity:
\[
R(\mathscr{G},\mathscr{H}):=\max \{r\in\mathbb{N}\cup\{\infty\}\,:\ B_\mathscr{G}(r)\simeq B_\mathscr{H}(r)\}.
\]
The quantity $B_\mathscr{G}(r)$ denotes the $r$-ball of $\mathscr{G}$, and the choice of root is irrelevant as the graph is transitive. Besides, isomorphisms are asked to map the chosen centre of $B_\mathscr{G}(r)$ to that of $B_\mathscr{H}(r)$. We define the \defini{local topology} as the unique metrisable topology on the space of isomorphism classes of transitive graphs satisfying
\[
\mathscr{G}_n\xrightarrow[n\to\infty]{}\mathscr{G}_{\infty}\quad \iff\quad R(\mathscr{G}_n,\mathscr{G}_\infty)\xrightarrow[n\to\infty]{}\infty.
\]
Around 2008, Schramm conjectured that for transitive graphs, the critical threshold $p_c(\mathscr{G})$ depends only on the local geometry of $\mathscr{G}$ in the sense that  $p_c(\mathscr{G}_n)\to p_c(\mathscr{G})$ as $\mathscr{G}_n\to \mathscr{G}$ in the local topology, provided that $\sup_n p_c(\mathscr{G}_n)<1$.
For transitive graphs of polynomial growth (meaning that the cardinality of the $r$-ball is upper bounded by a polynomial in $r$), the case of isotropic Euclidean lattices was known prior to the conjecture \cite{grimmettmarstrand}, then possibly anisotropic abelian Cayley graphs were treated \cite{MT17}, before all transitive graphs of polynomial growth were covered in \cite{cmtlocality}. As for graphs with fast growth, the first case to be understood was that of uniformly nonamenable graphs converging to a regular tree of degree at least 3 \cite{benjamini2011critical}, which was followed by the case of transitive graphs with exponential growth \cite{2018arXiv180808940H} and that of graphs where the heat kernel satisfies a certain uniform stretched-exponential upper bound \cite{HH21}. In the end, building upon several works \cite{supercritpoly, 2018arXiv180808940H, nonunimodulartom}, Easo and Hutchcroft established the conjecture in full generality \cite{easo-hutchcroft-locality}. Let us mention that the assumption ``$\sup_n p_c(\mathscr{G}_n)<1$'' can safely be replaced by the weaker condition $p_c(\mathscr{G}_n)<1$ for all $n$ large enough. Interestingly, locality of $p_c$ has also been studied in the context of finite transitive graphs \cite{E24}.

In this article, we are interested in how observables that carry more information about the behaviour of the model depend on the underlying graph $\mathscr{G}$. In this direction, one expects locality phenomena for observables beyond $p_c$, such as the percolation function $\theta$ away from criticality. As explained below, a form of locality for $\theta$ was announced in the parallel work \cite{easo-hutchcroft-locality}, and in this article we are interested in establishing stronger forms of locality.

The geometric setup of the present article concerns transitive graphs for which the cardinality of the ball of radius $r$ is superlinear in $r$ and bounded above by a polynomial in $r$. The set of all isomorphism classes of such graphs is denoted by $\mathfrak{G}$. For a transitive graph, having superlinear growth is equivalent to $p_c<1$; see \cite{EST24, DGRSY20}.
Our theorems embody the following message: {\it given some $\mathscr{G}\in\mathfrak{G}$, even if we know only a ball of sufficiently large radius rather than the entire graph, we can still deduce much about the supercritical regime of percolation on $\mathscr{G}$.} 

In the setup of $\mathfrak{G}$, a key result regarding the supercritical behaviour of the model and the locality of $p_c$ is the sharpness result stated in Theorem~\ref{thm:cmt-sharpness} below; see \cite{supercritpoly}. Given a graph $\mathscr{G}$, we write $V(\mathscr{G})$ for its set of vertices, $E(\mathscr{G})$ for its set of edges, and $d_\mathscr{G}$ for its graph distance. Given two vertices $u$ and $v$ in $\mathscr{G}$, we write $\{u\longleftrightarrow v\}$ for the event that $u$ and $v$ belong to the same cluster.

\begin{theorem}[Contreras--Martineau--Tassion, 2024]\label{thm:cmt-sharpness}
	Let $\mathscr{G}\in\mathfrak{G}$ and $p_0\in (p_c(\mathscr{G}),1]$. Then, there is a constant $c>0$ such that the following holds:
	\begin{equation}
	\forall p\in [p_0,1],\qquad \forall u,v\in V(\mathscr{G}),\qquad \mathbb{P}_{p,\mathscr{G}}(u\longleftrightarrow v \centernot\longleftrightarrow \infty) \le \exp(-c d_\mathscr{G}(u,v)).
	\end{equation}
\end{theorem}

\noindent The above bound is, in fact, sharp at the exponential scale. 

In recent decades, there has been growing interest in obtaining finitary versions of certain important results, such as Gromov's characterisation of Cayley graphs of polynomial growth \cite{gromov81} and Trofimov's subsequent characterisation of transitive graphs of polynomial growth \cite{trofimov85polynomial, woess91}; see \cite{shalomtao2010, BGT12, tesseratointon, tesseratointonnew} for finitary versions. Inspired by these works, in the present paper, we obtain the following finitary (a.k.a.\ local) version of Theorem~\ref{thm:cmt-sharpness}.

\begin{theorem}\label{thm:sharpness-local}
	Let $\mathscr{G}\in\mathfrak{G}$ and $p_0\in (p_c(\mathscr{G}),1]$. Then, there is a constant $c>0$ such that the set of all $\mathscr{H}\in\mathfrak{G}$ satisfying $p_c(\mathscr{H})<p_0$ and the following condition is a neighbourhood of $\mathscr{G}$:
	\begin{equation}
	\forall p\in [p_0,1],\qquad \forall u,v\in V(\mathscr{H}),\qquad \mathbb{P}_{p,\mathscr{H}}(u\longleftrightarrow v \centernot\longleftrightarrow \infty) \le \exp(-c d_\mathscr{H}(u,v)).
	\end{equation}
    ~\hfill\hyperlink{target:sharpness}{\footnotesize\color{mycolor}$\downarrow$ context and proof}
\end{theorem}

Likewise, a sharp upper bound on the probability that the cluster of the origin is finite but has cardinality at least $n$ was obtained in \cite[Theorem~1.2]{supercritpoly}. We prove that these estimates hold uniformly on a neighbourhood of $\mathscr{G}$ in $\mathfrak{G}$; see Theorem~\ref{thm:cluster size} in Section~\ref{sec:large-cluster}. 

Beyond the cluster size distribution, of particular interest is the isoperimetry of $\mathcal{C}_o$. In Theorem~\ref{thm:cluster repulsion}, we prove that when $\mathcal{C}_o$ is finite, the number of edges that touch both $\mathcal{C}_o$ and the infinite cluster has an exponential tail, uniformly on a neighbourhood of $\mathscr{G}$ in $\mathfrak{G}$. This extends an earlier result of Pete \cite{pete2008note} from the context of $\mathbb{Z}^d$ to that of $\mathfrak{G}$. Using this exponential cluster repulsion result together with the stretched exponential decay of the cluster size, and following the argument of Pete \cite{pete2008note}, we show that the isoperimetric profile of the infinite cluster roughly coincides with that of the original lattice. As established in \cite{pete2008note} in the case of $\mathbb{Z}^d$, this has important implications, such as that simple random walk on the largest cluster of a finite box $[-n,n]^d$ has, with high probability, $L^{\infty}$-mixing time $\Theta(n^2)$, and that the return probability $p_n(o,o)$ on the infinite cluster almost surely decays as $p_n(o,o) = O(n^{-d/2})$. 
We expect all these random walk results to carry over to the context of $\mathfrak{G}$, but verifying this is beyond the scope of the present paper.

Our techniques go beyond such results and can be used to prove convergence of diverse percolation-related quantities, including the function $\theta$.
It is classical in statistical mechanics to define a phase transition to be a parameter where an observable describing the macroscopic behaviour of the system fails to be analytic. For percolation theory, the standard choice is $\theta$, and indeed $\theta$ witnesses the phase transition of the model at $p_c$. It was recently proved in \cite{analyticity} that for the hypercubic lattice $\mathbb{Z}^d$, its cluster-density function $\theta_\mathscr{G}$ is analytic on $(p_c,1]$, thereby answering an old question of Kesten \cite{KestenAnalytic}. Our first result in this direction extends this work to all graphs in $\mathfrak{G}$, thus verifying that $\theta$ does not witness any other phase transitions.

\begin{theorem}\label{thm:analyticity-poly}
    For every $\mathscr{G}\in\mathfrak{G}$, the function $\theta_\mathscr{G}$ is analytic on $(p_c(\mathscr{G}),1]$.\hfill\hyperlink{target:analytic}{\footnotesize\color{mycolor}$\downarrow$ proof}
\end{theorem}

Theorem~\ref{thm:analyticity-poly} and the forthcoming Theorem~\ref{thm:analyticity-local} can be extended to observables of the form $\mathbb{E}_p[F(\mathcal{C}_o) \mathbbm{1}_{|\mathcal{C}_o|<\infty}]$, where $F$ is a function that grows subexponentially in the diameter of the cluster $\mathcal{C}_o$ of the origin $o$. Several natural observables that describe the macroscopic behaviour of the system can be expressed in this form, such as the \emph{(truncated) susceptibility}, the \emph{(finite) open clusters per vertex}, the \emph{truncated $k$ point function}
and the \emph{(non-truncated) $k$ point function}. We remark that the proof of analyticity of the latter relies on the uniqueness of the infinite cluster; see \cite{analyticity, PaSevAn} for further details and precise definitions.

On the one hand, analytic functions enjoy a form of local-to-global rigidity themselves, due to the uniqueness theorem; hence so does $\theta$. On the other hand, away from $p_c$, one expects the values of $\theta$ to depend only on the local geometry of the graph, which itself provides global geometric information on $\mathscr{G}$ due to finitary versions of Trofimov's theorem. In the following theorem, we refine Theorem~\ref{thm:analyticity-poly} by establishing a strong form of locality with respect to both the analytic extension of $\theta$ and the underlying graph. 

\begin{theorem}\label{thm:analyticity-local}
    Let $\mathscr{G}\in\mathfrak{G}$ and $p_0\in (p_c(\mathscr{G}),1]$.
    Then, there are constants $c,C>0$ and a connected open neighbourhood $U$ of $[p_0,1]$ such that for every $k$ large enough and every $\mathscr{H}\in \mathfrak{G}$ with $R(\mathscr{G},\mathscr{H})\geq k$ we have the following: there are unique holomorphic functions $f_\mathscr{H}$ (resp. $f_\mathscr{G}$) from $U$ to $\mathbb{C}$ agreeing with $\theta_\mathscr{H}$ (resp. $\theta_\mathscr{G}$) on $[p_0,1]$, and they satisfy $\|f_\mathscr{H}-f_\mathscr{G}\|_\infty <Ce^{-ck}$.\\
    \phantom{a}\hfill\hyperlink{target:analytic}{\footnotesize\color{mycolor}$\downarrow$ proof}
\end{theorem}

As a corollary, we obtain the following locality result for all derivatives of $\theta$.

\begin{corollary}\label{coro:kdiff}
    Let $\mathscr{G}_n\xrightarrow[n\to\infty]{}\mathscr{G}_\infty$ be a converging sequence of elements of $\mathfrak G$.
    Then, for every $k\ge 0$, the function $\theta$ is $\mathcal{C}^k$ and $\theta^{(k)}_{\mathscr{G}_n}$ converges to $\theta^{(k)}_{\mathscr{G}_\infty}$ uniformly on any compact subset of $(p_c(\mathscr{G}_\infty),1]$.
\end{corollary}

\begin{remark}
    A proof of Corollary~\ref{coro:kdiff} for $k=0$ was sketched by Easo and Hutchcroft in the parallel work \cite{easo-hutchcroft-locality} for all converging sequences of transitive graphs satisfying $p_c<1$. The corollary generalises this result to all values of $k$ in the case of transitive graphs of polynomial growth. Let us mention that even the differentiability of $\theta_\mathscr{G}$ for all $\mathscr{G}\in\mathfrak{G}$ is a new result.
\end{remark}

\begin{remark}
    Continuity of the map $p_c:\mathfrak{G}\to [0,1]$ was originally derived in \cite{cmtlocality} by using the main results of \cite{supercritpoly}.
    Our proof of Theorem~\ref{thm:analyticity-local} provides an alternative way to perform the derivation. Indeed, pointwise convergence of $\theta_{\mathscr{G}_n}$ to $\theta_{\mathscr{G}_{\infty}}$ on $(p_c(\mathscr{G}_\infty),1]$ implies upper semicontinuity of $p_c:\mathfrak{G}\to [0,1]$. Lower semicontinuity is an older result known to hold for all transitive graphs; see \cite[Section~14.2]{pete2017probability} or \cite[Section~1.2]{duminil2016new}.
\end{remark}

To establish Theorem~\ref{thm:analyticity-poly}, we broadly follow the main strategy of~\cite{analyticity}. 
We introduce the notion of an \emph{interface}, a novel form of boundary that satisfies several desirable probabilistic and combinatorial properties --- see Definition~\ref{defi:interface} and Lemma~\ref{lem:cutset}. 
In~\cite{analyticity}, interfaces were formally defined on a different graph, namely the renormalised lattice $N\mathbb{Z}^d$. 
Some care is required when handling interfaces in the more general context of transitive graphs of polynomial growth, as a naive extension of the definition from $\mathbb{Z}^d$ can lead to several difficulties. 
Our approach instead defines interfaces directly on the original graph, inspired by the technique of coarse-graining while remaining within the same graph from~\cite[Section~10]{supercritpoly}. 

The definition of interfaces is specifically designed to yield an \emph{exact}\footnote{Peierls' argument provides an upper bound for~$1-\theta$. 
Our method can be viewed as a refinement of Peierls' argument that yields an exact equality instead, which is crucial for proving analyticity.} 
expression for $\theta$ as a series of polynomials over collections of interfaces, called \emph{multi-interfaces}, using the inclusion--exclusion principle. 
The main challenge then lies in establishing uniform convergence of this series in some suitable domain in $\mathbb{C}$. 
As a key ingredient, we use the fact that, by the non-overlapping property of interfaces, we may restrict to multi-interfaces consisting of disjoint interfaces. 
For the convergence, a crucial input is the local uniqueness result from~\cite{supercritpoly}, namely their Proposition~1.3. 

The additional ingredient needed to pass from Theorem~\ref{thm:cmt-sharpness} to Theorem~\ref{thm:sharpness-local}, 
and from Theorem~\ref{thm:analyticity-poly} to Theorem~\ref{thm:analyticity-local}, 
is Theorem~\ref{thm:finitary-timar} below. 
To state it properly, we first provide some definitions and context.
In the planar setting, Peierls' argument exploits the fact that finite clusters are enclosed by circuits in the dual graph. To extend the argument beyond the planar case, circuits are replaced by minimal cutsets.

Let $\mathscr{G}$ be a graph. Let $u, v \in V(\mathscr{G})$ and $\Pi \subset V(\mathscr{G})$. We say that $\Pi$ is a \defini{cutset between $u$ and $v$} if any path from $u$ to $v$ has to pass through at least one vertex of $\Pi$. We say that $\Pi$ is a \defini{minimal cutset between $u$ and $v$} if it is minimal for inclusion among all cutsets between $u$ and $v$. Likewise, by taking $\Pi\subset E(\mathscr{G})$, we can define the notions of \defini{bond-cutsets} and \defini{minimal bond-cutsets}.
For $k\geq 1$ and $\Pi\subset V(\mathscr{G})$, we say that $\Pi$ is \defini{$k$-coarse connected} if it is connected as a subset of $\left(V,E_{ k}\right)$, where
\[
\{u,v\}\in E_{ k}\quad \iff \quad 1\le d_{\mathscr{G}}(u,v)\le k.
\]
It follows from works of Babson--Benjamini and Tim\'ar \cite{babson1999cut, timarcutset} that for every $\mathscr{G}\in\mathfrak{G}$, there is a finite constant $c$ such that for all vertices $u$ and $v$, all minimal cutsets between $u$ and $v$ are $c$-coarse connected; see \cite[Lemma~2.1]{supercritpoly}. Our next theorem is a local version of this result for finite minimal cutsets.

\begin{theorem}
    \label{thm:finitary-timar}
    For every $\mathscr{G}\in\mathfrak{G}$, there is some $c<\infty$ such that the set of all $\mathscr{H}\in\mathfrak{G}$ satisfying the following condition is a neighbourhood of $\mathscr{G}$: for any two vertices $u,v\in V(\mathscr{H})$, every finite minimal cutset from $u$ to $v$ is $c$-coarse connected.\hfill\hyperlink{target:timar}{\footnotesize\color{mycolor}$\downarrow$ proof}
\end{theorem}

\noindent Along the way of proving Theorem~\ref{thm:finitary-timar}, we establish results on the connectivity of minimal cutsets that we believe are of independent interest and, in particular, we answer Question~2 from \cite{babson1999cut}; see Section~\ref{sec:qi}. 

Let us also remark that the finiteness of the constant $c$ in the above theorem, for a fixed graph $\mathscr{G} \in \mathfrak{G}$, follows from the fact that the cycle space of $\mathscr{G}$ is generated by cycles of bounded length. However, it is possible that some basic cycles in $\mathscr{H}$ are much longer than any basic cycle in $\mathscr{G}$, regardless of how close $\mathscr{G}$ and $\mathscr{H}$ are. Our approach avoids such pathologies by comparing the connectivity of minimal cutsets in $\mathscr{H}$ and $\mathscr{G}$ directly.

\begin{remark}
It is a standard fact that, on connected graphs of bounded degree, coarse connectivity of cutsets implies an exponential upper bound on the number of cutsets of a given size that separate a fixed origin from infinity. Our Theorem~\ref{thm:finitary-timar} implies that this exponential upper bound is local for graphs in $\mathfrak{G}$. The latter can also be deduced from the results of \cite{EST24}, namely from the fact that $\theta_{\mathscr{G}}(1-\varepsilon)\geq \varepsilon$ for some uniform $\varepsilon>0$, which follows from the proof of the so-called gap at $1$ for $p_c$, and their Theorem $1$. We remark that when combined with the results of \cite{supercritpoly}, this is sufficient to prove Theorem~\ref{thm:sharpness-local}. However, for the proof of the analyticity of $\theta$, it is crucial that the coarse connectivity of cutsets is itself local.
\end{remark}

Let us now come back to percolation and end this introduction by focusing once again on the properties of the percolation function $\theta$. One can formulate a family of interesting problems in percolation theory by investigating the regularity of $\theta$. Depending on the meaning one gives to the word ``regularity'', and depending on the graphs one considers, this can give rise to diverse mathematical problems. When one fixes a graph $\mathscr{G}$, the continuity of $\theta$ at $p_c(\mathscr{G})$ is equivalent to the absence of percolation at criticality, i.e.\ whether $\theta(p_c)=0$, a famous open problem which has been studied extensively in the case of the hypercubic lattice $\mathbb{Z}^d$. Continuity at $p_c$ is known in dimension $d=2$ \cite{KestenCritical} and in dimensions $d\geq 11$ \cite{HS90, FH17}, but remains open in all intermediate dimensions, with the case of dimension $d=3$ being notoriously difficult. See \cite{BLPS99, Timar06, PPS06, slabs, H16, HH21} for related works beyond the context of $\mathbb{Z}^d$. On the other hand, showing continuity of $\theta$ on $[0,1]\setminus \{p_c\}$ is much more tractable; see \cite{schonmann1999stability} for a proof holding for all transitive graphs.

In another direction, one can study the regularity of the map $\mathscr{G}\longmapsto \theta_{\mathscr{G}}$, where we now view $\theta$ as a function defined on the whole interval $[0,1]$. This turns out to be related to the aforementioned problem of whether $\theta(p_c)=0$. To see this, let us introduce some notation. Given two graphs $\mathscr{G}_1$ and $\mathscr{G}_2$, we set their product to have vertex-set $V(\mathscr{G}_1)\times V(\mathscr{G}_2)$ and declare $(v_1,v_2)$ to be adjacent to $(v'_1,v'_2)$ if and only if there is some $i$ such that $v_i$ is adjacent to $v'_i$ and, for the other value $j=2-i$, we have $v_j=v'_j$. We say that a (non-empty connected) graph is a cycle-graph if all its vertices have degree 2: these graphs are either a cycle of finite length at least 3, or a bi-infinite line. \label{page:ab}Let $\Ab$ be the set of all isomorphism classes of graphs that are the product of finitely many cycle-graphs, at least two of which are infinite. In particular, graphs in $\Ab$ contain a copy of the lattice $\mathbb{Z}^d$ for some $d\geq 2$. We endow $\Ab$ with the induced topology it inherits as a subset of $\mathfrak{G}$.

We are now able to state the following proposition, which is in the same spirit as \cite[Proposition~3]{slabs}. Equivalence of the first two items is classical but we restate it for emphasis.

\begin{proposition}
\label{prop:equiv}
    The following statements are equivalent:
    \begin{enumerate}
        \item\label{item:1} every graph in $\Ab$ satisfies $\theta(p_c)=0$,
        \item\label{item:2} for every $\mathscr{G}\in \Ab$, the map $p\longmapsto \theta_\mathscr{G}(p)$ is continuous from $[0,1]$ to $[0,1]$,
        \item\label{item:3} for every $p\in[0,1]$, the map $\mathscr{G}\longmapsto \theta_\mathscr{G}(p)$ is continuous from $\Ab$ to $[0,1]$,
        \item\label{item:4} the map $(\mathscr{G},p)\longmapsto \theta_\mathscr{G}(p)$ is continuous from $\Ab\times [0,1]$ to $[0,1]$,
        \item\label{item:5} the map $\mathscr{G}\longmapsto \theta_\mathscr{G}$ is continuous from $\Ab$ to $\mathcal{L}^\infty([0,1],\mathbb{R})$,
        \item\label{item:6} the map $\mathscr{G}\longmapsto \theta_\mathscr{G}$ is well defined and continuous from $\Ab$ to $\mathcal{C}([0,1],\mathbb{R})$ endowed with the uniform topology.~\hfill\hyperlink{target:equiv}{\footnotesize\color{mycolor}$\downarrow$ context and proof}
    \end{enumerate}
\end{proposition}

Problems of regularity of the map $\mathscr{G}\longmapsto\theta_\mathscr{G}$ display diversity. Results such as Theorem~\ref{thm:analyticity-local} and Corollary~\ref{coro:kdiff} are true and proven. In the setup $\Ab$, continuity of $\mathscr{G}\longmapsto \theta_\mathscr{G}(p)$ holding for all $p$ is equivalent to a famous conjecture. And yet another statement is expected not to hold, namely uniform convergence of the derivative $\theta_{\mathscr{G}}'$ on $(p_c,1]$; see Section~\ref{sec:theta-near-pc} for precise statements. 

\paragraph{Structure of the paper} In Section~\ref{sec:connecticut}, we prove Theorem~\ref{thm:finitary-timar} on coarse connectedness of finite minimal cutsets. We use this result in Section~\ref{sec:analytic} to prove Theorem~\ref{thm:analyticity-local} and in Section~\ref{sec:strong-proba-estimates} to establish good probabilistic upper bounds in the supercritical regime, which include Theorem~\ref{thm:sharpness-local}. We prove Proposition~\ref{prop:equiv} in Section~\ref{sec:equiv}. At last, in Section~\ref{sec:theta-near-pc}, we make comments regarding $\theta_{\mathscr{G}}(p)$ and $\theta'_{\mathscr{G}}(p)$ when the graph $\mathscr{G}$ is variable and $p$ is very slightly supercritical. All sections can be read almost independently of each other.

\paragraph{Acknowledgements} We thank Vincent Tassion for an insightful conversation that inspired the content of Section~\ref{sec:theta-near-pc}. We are grateful to G\'abor Pete for suggesting to include a proof of local forms of exponential cluster repulsion and anchored-isoperimetric estimates. We thank Franco Severo for pointing out that \cite{EST24} gives a local bound on the number of minimal cutsets. At last, we are grateful to Nathan Deloire, Alexander Glazman, Tom Hutchcroft, \'Ad\'am Tim\'ar, and Matthew Tointon for helpful discussions. CP was supported by an EPSRC New Investigator Award (UKRI1019).

\hypertarget{target:timar}{
\section{Minimal cutsets are coarsely connected}}
\label{sec:connecticut}

This purely geometric section is dedicated to proving a result on the connectivity of cutsets, namely Theorem~\ref{thm:loc-bded}. This result is a crucial ingredient in the proof of the main theorems. It enables us, whenever we have a converging sequence of graphs in $\mathfrak G$, to obtain \emph{uniform} control on the finite cutsets of these graphs.

\subsection{Definitions and statement of Theorem~\ref{thm:loc-bded}}

We denote by $\pei(\mathscr{G})$ the infimum of all $k\geq 1$ such that, for every $u,v\in V$, every \defini{minimal} cutset between $u$ and $v$ is $k$-coarse connected. If no such $k$ exists, we set $\pei(\mathscr{G})=\infty$. This quantity has been studied in \cite{babson1999cut,timarcutset, timar2013boundary}, and it will play an important role in this paper. 
We will also need a variant of this quantity. We denote by $\weak(\mathscr{G})$ the infimum of all $k\geq 1$ such that, for every $u,v\in V$, every \defini{finite minimal} cutset between $u$ and $v$ is $k$-coarse connected. Again, if no such $k$ exists, we set $\weak(\mathscr{G})=\infty$. 

\begin{remark}
    \label{rem:connecticut}
    There exist graphs $\scr G$ such that both $\pei(\mathscr{G})$ and $\weak(\mathscr{G})$ are finite, yet their values differ. For example, consider the square lattice and let $\scr G$ be the graph obtained by removing a large ball around the origin. Then $\weak(\mathscr{G})=2$, as for the square lattice, while $\pei(\mathscr{G})$ is much larger. To see the latter, consider the vertices of the horizontal axis lying in $\mathscr{G}$, and note that they form an infinite minimal cutset between any vertex in the upper half-plane and any vertex in the bottom half-plane.
\end{remark}

The purpose of this section is to prove the following theorem.

\begin{theorem}[reformulation of Theorem~\ref{thm:finitary-timar}]
\label{thm:loc-bded}
The map $\weak$, defined on $\mathfrak G$, takes only finite values and is locally bounded.
In other words, for every $\mathscr{G}\in \mathfrak{G}$, there are finite constants $r=r(\mathscr{G})\geq 0$ and $t=t(\mathscr{G})\geq 0$ such that for every $\mathscr{H}\in\mathfrak{G}$, we have
\[
R(\mathscr{G},\mathscr{H})\ge r\implies \weak(\mathscr{H})\le t.
\]
\end{theorem}
We remark that the finiteness of $\weak$ on $\mathfrak G$ follows from the fact that the cycle space of these graphs is generated by cycles of bounded length. However, it is possible that some basic cycles in $\mathscr{H}$ are much longer than any basic cycle in $\mathscr{G}$, regardless of how close $\mathscr{G}$ and $\mathscr{H}$ are. For example, this occurs for the ``slab'' $\mathbb{Z}^2\times \mathbb{Z}/n$ and $\mathbb{Z}^3$. For this reason, our strategy relies on a more geometric approach, where we compare the values of  $\weak(\mathscr{H})$ and $\weak(\mathscr{G})$ directly.

A useful property to keep in mind, while working with the class $\mathfrak G$, is that every transitive graph of polynomial growth is quasi-isometric to a Cayley graph of some nilpotent group \cite{trofimov85polynomial, woess91}. The proof of Theorem~\ref{thm:loc-bded}
 crucially relies on a finitary refinement of this result, recently obtained by Tessera and Tointon  \cite{tesseratointon, tesseratointonnew}. This refinement allows us to reduce our problem to understanding the behaviour of $\weak$ for Cayley graphs of nilpotent groups, and how it is affected by quasi-isometries.

To handle nilpotent Cayley graphs, we use the fact that, given a sequence of Cayley graphs converging to a Cayley graph of a nilpotent group, all but finitely many graphs in the sequence are quotients of the limit graph. For this reason, Section~\ref{sec:quotient} begins with a proposition that controls $\weak(\mathscr{G})$ for one-ended graphs $\mathscr{G}$ which are quotients of a nice graph $\mathscr{H}$. We need a detour via bond-versions of our main quantities (see Section~\ref{sec:bond}) and, interestingly, the upper bound we get for $\weak_E(\mathscr{G})$ is $\pei_E(\mathscr{H})$, not $\weak_E(\mathscr{H})$. This is why both notions of ``cutconnectivity'' --- with and without a star --- need to be introduced, which constitutes the main subtlety of the present Section~\ref{sec:connecticut}. The behaviour of $\weak$ under quasi-isometries is then investigated in Section~\ref{sec:qi}. With these results at hand, we finally establish Theorem~\ref{thm:loc-bded} in Section~\ref{sec:loc-bded}.

\subsection{Cutsets and bond-cutsets}
\label{sec:bond}

As a tool to understand $\pei(\mathscr{G})$ and $\weak(\mathscr{G})$, we also introduce edge-versions of these quantities. We say that two distinct edges $e$ and $e'$ are $k$-adjacent if there is one endpoint of $e$ that lies at distance at most $k$ from an endpoint of $e'$. We can thus define $\pei_E(\mathscr{G})$ and $\weak_E(\mathscr{G})$ as above, by considering bond-cutsets instead of cutsets.

\begin{remark}
    A key-difference between cutsets and bond-cutsets goes as follows: as removing an edge cannot increase the number of connected components by more than one, minimal bond-cutsets always break the original graph into exactly two connected components. This property does not hold when cutsets are taken relative to vertices.
\end{remark}

In what follows, given a connected subgraph $D$ of $\mathscr{G}$, we write  $\partial_E D$ for the set of edges with one endpoint in $D$ and the other not in $D$.

\begin{lemma}
    \label{lem:bond}
    For every graph $\mathscr{G}$, we have $\weak(\mathscr{G})\le \weak_E(\mathscr{G})+2$.
\end{lemma}

\begin{proof}
Let $\Pi$ be a finite cutset between two vertices $u$ and $v$. We will show that $\Pi$ is $(\weak_E(\mathscr{G})+2)$-coarse connected. If $\Pi=\{u\}$ or $\Pi=\{v\}$, then there is nothing to prove, so let us assume that $\Pi\neq \{u\},\{v\}$. 

Let $D$ be the connected component of $u$ in the complement of $\Pi$. Note that $\partial_E D$ is a finite bond-cutset that separates $u$ and $v$. Furthermore, the minimality of $\Pi$ implies that $\partial_E D$ is a minimal bond-cutset, and that each vertex in $\Pi$ is an endpoint of some edge in $\partial_E D$. Hence $\partial_E D$ is $\weak_E(\mathscr{G})$-coarse connected, and by the triangle inequality, $\Pi$ is $(\weak_E(\mathscr{G})+2)$-coarse connected, as desired.
\end{proof}
\begin{remark}
    We shall not need it but the inequality $\pei(\mathscr{G})\le \pei_E(\mathscr{G})+2$ holds as well, for the same reason.
\end{remark}

\subsection{Cutconnectivity and fibrations}
\label{sec:quotient}

\newcommand{\La}{\mathscr{L}}
\newcommand{\Sm}{\mathscr{S}}

Given two graphs $\Sm$ and $\La$, we say that a map $\pi:V(\La)\to V(\Sm)$ is \defini{1-Lipschitz} if, for all $x,y\in V(\La)$, we have $d_{\Sm}(\pi(x),\pi(y))\leq d_{\La}(x,y)$; this is equivalent to asking that for any two neighbouring vertices $x,y$ in $\La$, their images $\pi(x)$ and $\pi(y)$ either coincide or are adjacent. We call a map $\pi: V(\La) \to V(\Sm)$ a \defini{fibration} from $\La$ to $\Sm$ if it is 1-Lipschitz and satisfies that, for every vertex $x$ of $\La$ and every neighbour $v$ of $\pi(x)$, there exists a neighbour $y$ of $x$ such that $\pi(y) = v$.

\begin{remark}
    We have chosen the letters $\Sm$ and $\La$ to keep in mind which graph is Small and which one is Large. Notice that, as we have taken our graphs to be non-empty and connected, fibrations are automatically surjective. Indeed, consider a vertex $o$ in $\La$. For every vertex $x$ in $\Sm$, we can find an element in $\pi^{-1}(\{x\})$ by lifting a path from $\pi(o)$ to $x$.
\end{remark}

Recall that a graph is said to be \defini{one-ended} if, for every finite set of vertices $F$, its complement consists in exactly one infinite connected component and zero, one or several finite connected components.

\begin{proposition}
\label{prop:mon}
        Let $\pi:V(\La)\to V(\Sm)$ be a fibration. Assume that $\Sm$ is one-ended and that, for every finite subset $F$ of $V(\Sm)$, the infinite connected component $K$ of the complement of $F$ satisfies that the graph induced by $\La$ on $\pi^{-1}(K)$ is connected.
        
        Then, we have $\weak_E(\Sm)\le \pei_E(\La)$.
\end{proposition}

\begin{proof}
    Let $\Pi$ be a finite minimal bond-cutset between two vertices $u$ and $v$ in $\Sm$. We wish to prove that $\Pi$ is $\pei_E(\La)$-coarse connected.  As $\Pi$ is a finite minimal bond-cutset and because $\Sm$ is one-ended, the graph $\Sm\setminus\Pi$ must consist of exactly two connected components: a finite one, with vertex-set $F$, and an infinite one,  with vertex-set $K$. This is where we use the fact that we work with bond-cutsets instead of cutsets. Notice that $\partial_E F=\Pi$ and $\partial_E K=\Pi$. As every edge of $\Pi$ has an endpoint in $F$ and the other in $K$, the graph structure induced by $\Sm$ and by $\Sm\setminus \Pi$ on $F$ are the same, and so is the case for $K$.

    Without loss of generality, we may assume that $u$ belongs to $F$ and that $v$ belongs to $K$. Since $\pi$ is surjective, we can pick $x$ in $\pi^{-1}(\{u\})$ and $y\in \pi^{-1}(\{v\})$. Write $\pi^{-1}(\Pi)$ for the set of all edges $\{z_1,z_2\}$ in $\La$ such that $\{\pi(z_1),\pi(z_2)\}$ is an edge that belongs to $\Pi$. As $\Pi$ is a bond-cutset between $u$ and $v$ and since $\pi$ is 1-Lipschitz, the set $\pi^{-1}(\Pi)$ is a bond-cutset between $x$ and $y$. Let $D$ denote the connected component of $x$ in $\La\setminus \pi^{-1}(\Pi)$, which cannot contain $y$. Now, let $\tilde \Pi:=\partial_E D$. This set may well be finite or infinite --- see Figure~\ref{fig:fibration-connecticut}. Note that $\tilde \Pi$ is a bond-cutset between $x$ and $y$. 
    Because of our assumption, we further know that $\pi^{-1}(K)$ is connected in $\La$, and since each edge of $\pi^{-1}(\Pi)$ has one endpoint not in $\pi^{-1}(K)$, $\pi^{-1}(K)$ is, in fact, connected in $\La\setminus \pi^{-1}(\Pi)$. Moreover, since $\tilde \Pi\subset \pi^{-1}(\Pi)$, every edge of $\tilde\Pi$ contains a vertex of $\pi^{-1}(K)$. Therefore, for every $e\in \tilde \Pi$, there is a path from an endpoint of $e$ to $y$ that avoids $\pi^{-1}(\Pi)$. By the definition of $D$ and $\tilde \Pi$, the same holds for $x$ instead of $y$. As a result, $\tilde \Pi$ is a \emph{minimal} bond-cutset between $x$ and $y$.

\begin{figure}
    \centering
    \includegraphics[width=14cm]{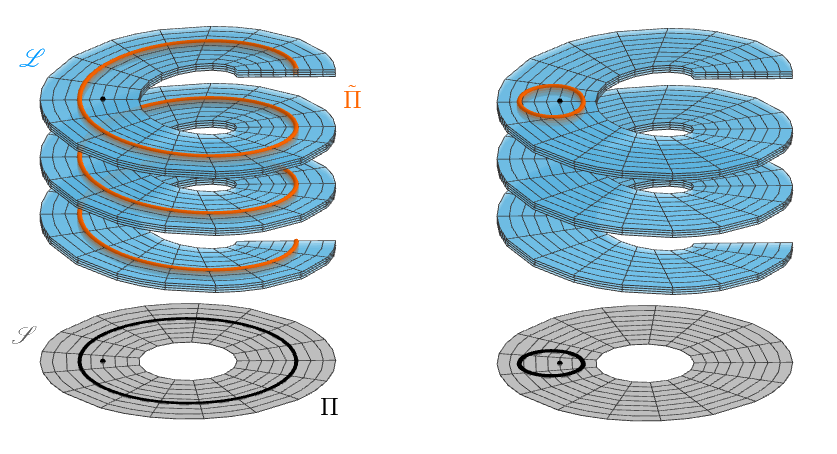}
    \caption{Illustration of $\tilde \Pi$.}
    \label{fig:fibration-connecticut}
\end{figure}

    Since $\tilde \Pi$ is a minimal bond-cutset between two vertices of $\La$, it is $\pei_E(\La)$-coarse connected. Moreover, as the map $\pi$ is 1-Lipschitz, the set $\pi(\tilde \Pi)$ must thus be $\pei_E(\La)$-coarse connected as well.\footnote{The set $\pi(\tilde \Pi)$ is indeed well defined as a set of edges because $\tilde\Pi\subset \pi^{-1}(\Pi)$.} By definition, we know that $\pi(\tilde \Pi)\subset \Pi$. In order to conclude, it thus suffices to prove the reverse inclusion.
    As $\Pi$ is a minimal bond-cutset between $u$ and $v$, for every $e\in \Pi$, there is a path from $u$ to an endpoint $w$ of $e$ that avoids $\Pi$. Using the fact that $\pi$ is a fibration, we can lift this path to get a path from $x$ to some $w'\in \pi^{-1}(\{w\})$ that avoids $\pi^{-1}(\Pi)$, i.e.\ $w'\in D$. Considering the other endpoint $z$ of $e$, and using again that $\pi$ is a fibration, we obtain that $w'$ has a neighbour $z'\in \pi^{-1}(z)$ which is not in $D$, i.e.\ $\{w',z'\}\in \tilde \Pi$. This means that $e\in \pi(\tilde \Pi)$, and establishes the inclusion $\Pi\subset \pi(\tilde\Pi)$, thus concluding the proof.
\end{proof}

\begin{remark}
    Proposition~\ref{prop:mon} is the reason why we need to introduce both $\pei_E$ and $\weak_E$. Indeed, its statement involves both quantities (for different graphs) and our argument does not directly imply\footnote{at least not without modifying our assumptions, which we need to be satisfied by converging sequences of Cayley graphs of one-ended nilpotent groups} that $\pei_E(\Sm)\le \pei_E(\La)$ or $\weak_E(\Sm)\le \weak_E(\La)$.
\end{remark}

\begin{remark}
    The assumption of connectedness of $\pi^{-1}(K)$ plays a crucial role in the proof of Proposition~\ref{prop:mon}. Indeed, if we try to connect a suitable endpoint $z$ of an arbitrary edge $e\in \tilde \Pi$ to $y$ outside $\pi^{-1}(\Pi)$ by using lifts as in the last stage of the proof, we obtain two paths but run into troubles when trying to connect them. We may obtain a path from $x$ to $z$ and then a path from $z$ to some well chosen $y'\in \pi^{-1}(\{v\})$, but there is absolutely no reason why $y'$ should coincide with $y$. Likewise, we may obtain a path from $x$ to $z\in\pi^{-1}(\{w\})$ and then a path from $y$ to some well chosen $z'\in \pi^{-1}(\{w\})$, but then there is no reason why $z'$ should coincide with $z$.
\end{remark}

The following lemma will be useful for verifying that the assumptions of Proposition~\ref{prop:mon} hold in our concrete situation of interest.

\begin{lemma}
\label{lem:mon}
Let $\pi:V(\La) \to V(\Sm)$ be a fibration. Assume that $\Sm$ is one-ended and that, for every finite subset $F$ of $V(\Sm)$, there is a vertex $u$ in $\Sm$ such that any two vertices in $\pi^{-1}(\{u\})$ can be connected by a path that never intersects $\pi^{-1}(F)$.        

        Then, for every finite subset $F$ of $V(\Sm)$, the infinite connected component $K$ of the complement of $F$ satisfies that the graph induced by $\La$ on $\pi^{-1}(K)$ is connected.
\end{lemma}

\begin{proof}
    Let $F$ be a finite subset of $V(\Sm)$. Denote by $K$ the infinite connected component of the complement of $F$. Notice that, since $\Sm$ is one-ended, $K$ is uniquely defined and that its complement $F'$ is a finite set containing $F$. Applying the assumption to $F'$, we get a vertex $u$ such that any two vertices in $\pi^{-1}(\{u\})$ can be connected by a path that never intersects $\pi^{-1}(F')$. In particular, $u$ belongs to $K$ and any two vertices in $\pi^{-1}(\{u\})$ can be connected by a path that never intersects $\pi^{-1}(F)$.

    Let us now pick two vertices $x$ and $y$ in $\pi^{-1}(K)$. As the vertices $u$, $\pi(x)$, and $\pi(y)$ belong to the connected graph $K$, we can find a path from $\pi(x)$ to $u$ and a path from $u$ to $\pi(y)$. Since $\pi$ is a fibration, we can lift these paths to get two paths avoiding $\pi^{-1}(F)$: one path from $x$ to some vertex in $z_1\in \pi^{-1}(\{u\})$ and another path from some vertex $y$ to some $z_2\in \pi^{-1}(\{u\})$. To conclude, it remains to connect $z_1$ and $z_2$ by a path that avoids $\pi^{-1}(F)$, which is possible because of our choice of $u$.
\end{proof}

We are now able to state and prove a corollary of Proposition~\ref{prop:mon} that will be directly used in our proof of Theorem~\ref{thm:loc-bded}. In order to state it in the natural generality of the argument, recall that a finitely generated group $G$ is \defini{one-ended} if any Cayley graph of $G$ with respect to a finite generating set is one-ended, and that it is \defini{Noetherian} if all its subgroups are finitely generated. Also recall that every finitely generated nilpotent group is Noetherian \cite[Theorem~13.57]{drutu-kapovitch}, and that finitely generated nilpotent groups of superlinear growth are always one-ended: these two facts guarantee that Corollary~\ref{coro:mon} will cover all our situations of interest.

\begin{corollary}
\label{coro:mon}
Let $G$ and $H$ be two finitely generated groups, with $G$ one-ended and $H$ Noetherian. Let $\pi:H\to G$ be a surjective group homomorphism, and let $S$ be a finite generating subset of $H$. Let $\mathscr{H}$ denote the Cayley graph of $H$ with respect to $S$ and $\mathscr{G}$ denote the Cayley graph of $G$ with respect to $\pi(S)$.

Then, we have $\weak_E(\mathscr{G})\le \pei_E(\mathscr{H})$.
\end{corollary}

\begin{proof}
    First, note that $\pi$ is a fibration from $\scr{H}$ to $\scr{G}$. By Proposition~\ref{prop:mon} and Lemma~\ref{lem:mon}, it thus suffices to check that, for every finite subset $F$ of $V(\mathscr{G})$, there is a vertex $u$ in $\scr G$ such that any two vertices in $\pi^{-1}(\{u\})$ can be connected by a path that never intersects $\pi^{-1}(F)$. 

   If $\ker(\pi)$ is the trivial group that consists only of the identity element, then there is nothing to prove. Thus, we assume that $\ker(\pi)$ is not the trivial group. As $H$ is Noetherian, $\ker(\pi)$ is a finitely generated group. We can therefore pick a non-empty finite set $T\subset \ker(\pi)$ that generates $\ker(\pi)$ as a group. Let $c:=\max_{t\in T} d_{\scr H}(1,t)$, where $1$ denotes the identity element of $H$. 

    Let $F$ be a finite subset of $V(\mathscr{G})$. The graph $\scr G$ being one-ended, it has in particular infinite diameter. We can therefore pick some vertex $u\in V(\mathscr{G})$ at a distance strictly greater than $c$ from $F$. Let $x$ and $y$ be two elements in $\pi^{-1}(\{u\})$. We can write $y=xk$ for some $k\in \ker(\pi)$. In turn, we can write $k=t_1\dots t_n$, where each $t_i$ belongs to $T\cup T^{-1}$. Finally, each $t_i$ can be written as $s_{i,1}\dots s_{i,m(i)}$, where each $s_{i,j}$ belongs to $S\cup S^{-1}$ and $m(i)\le c$. Putting all this together, we write $k$ as a word with letters $s_{i,j}$. Starting with the empty word and adding the letters $s_{i,j}$ one by one to the right yields a path from $1$ to $k$ in $\scr{G}$. Multiplying this path to the left by $x$ produces a path $\gamma$ from $x$ to $y$.

    We claim that $\gamma$ does not intersect $\pi^{-1}(F)$. To see this, first observe that all intermediate vertices of the form $xt_1\dots t_i$ visited by $\gamma$ are mapped to $u$ under $\pi$, and therefore cannot belong to $\pi^{-1}(F)$. As for the remaining vertices visited by $\gamma$, they are of the form $xt_1\dots t_i s_{i,1}\dots s_{i,j}$, which also cannot lie in $\pi^{-1}(F)$: indeed, their image under $\pi$ is at distance at most $j\le c$ from $u$, whereas $u$ was taken so that its distance from $F$ is larger than $c$. This completes the proof.
\end{proof}

\subsection{Cutconnectivity and quasi-isometries}
\label{sec:qi}
\newcommand{\fat}{{C''}}
\newcommand{\fatty}{{C'}}

Given $A\ge 1$, say that a function $\varphi:V(\mathscr{G})\rightarrow V(\mathscr{H})$ is an \defini{$A$-quasi-isometry} if the following two conditions hold:
\begin{enumerate}
    \item for all $x,y \in V(\mathscr{G})$, we have 
    $$\tfrac{1}{A} d_{\mathscr{G}}(x,y)-A\leq d_{\mathscr{H}}(\varphi(x),\varphi(y))\leq A d_{\mathscr{G}}(x,y)+A,$$
    \item for every $z\in V(\mathscr{H})$, there is some  $x\in V(\mathscr{G})$ such that $d_{\mathscr{H}}(z,\varphi(x))\leq A$.
\end{enumerate}
We say that the graphs $\mathscr{G}$ and $\mathscr{H}$ are \defini{$A$-quasi-isometric} if there is a pair of $A$-quasi-isometries $\varphi:V(\mathscr{G})\rightarrow V(\mathscr{H})$ and $\psi:V(\mathscr{H})\rightarrow V(\mathscr{G})$ such that for every $x\in V(\mathscr{G})$ and every $z\in V(\mathscr{H})$, we have $d_{\mathscr{G}}(\psi\circ\varphi (x),x)\leq A$ and $d_{\mathscr{H}}(\varphi\circ\psi (z),z)\leq A$. We call $\varphi$ and $\psi$ $A$-quasi-inverses of each other. The existence of an $A$-quasi-isometry implies that $\mathscr{G}$ and $\mathscr{H}$ are $B$-quasi-isometric for some constant $B$ that depends only on $A$.

\begin{proposition}
\label{prop:qi}
    For all $A,B\ge 1$, there is some $C=C(A,B)\ge0$ such that the following holds.
    Let $\mathscr{G}$ and $\mathscr{H}$ be two graphs. Assume that $\mathscr{G}$ and $\mathscr{H}$ are $A$-quasi-isometric and that $\weak(\mathscr{G})\le B$. Then, we have $\weak(\mathscr{H})\le C$.
\end{proposition}

\begin{remark}
    Proposition~\ref{prop:qi} also holds, with the same proof, if both occurrences of $\weak$ are replaced with $\pei$. This answers Question~2 from \cite{babson1999cut}, and further does so by providing uniformity for the constants involved. For several years, Question~2 had been considered solved \cite[Theorem 3.2]{timarcutset} but a gap in the proof was later found \cite[Remark~8.4]{brieussel-gournay}. The mistake is explained in Figure~\ref{fig:qi-timar}.
\end{remark}

\begin{figure}[h!]
    \centering
    \includegraphics[width=0.8\textwidth]{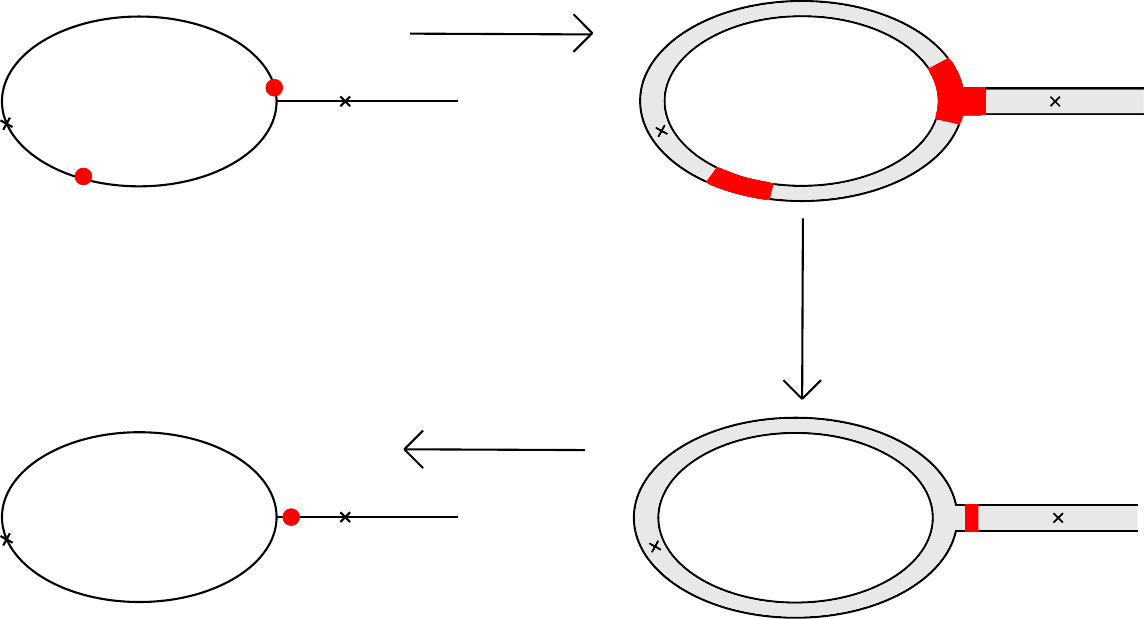}
    \put(-208,192){$\psi$}
    \put(-208,60){$\varphi$}
    \put(-320,155){{\color{red} $\Pi$}}
    \put(-372,153){$x$}
    \put(-257,170){$y$}
    \put(-35,175){$\psi(y)$}
    \put(-183,152){$\psi(x)$}
    \put(-35,43){$\psi(y)$}
    \put(-185,20){$\psi(x)$}
    \put(-372,23){$x$}
    \put(-257,38){$y$}
    \caption{In the top left corner, we depict the vertices $x$ and $y$ in the graph $\mathscr{H}$ as crosses, and the vertices of $\Pi$ as dots. In the top right corner, we depict the image under $\psi$, where we have highlighted the $\fat$-neighbourhood of $\psi(\Pi)$ in red. In the bottom right corner, we depict a minimal cutset between $\psi(x)$ and $\psi(y)$ in the $\fat$-neighbourhood of $\psi(\Pi)$, and in the bottom left corner, its image under $\varphi$. The latter gives no information on the connectivity of $\Pi$.}
    \label{fig:qi-timar}
\end{figure}

\begin{figure}[h!]
        \centering
        \includegraphics[width=0.9\linewidth]{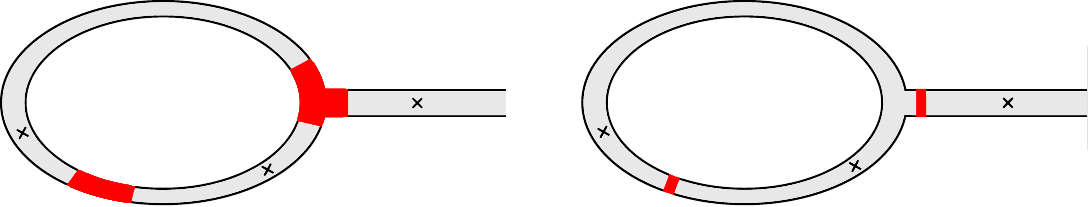}
        \put(-432,22){$\psi(x)$}
        \put(-260,50){$\psi(y)$}
        \put(-310,2){$v$}
        \put(-212,22){$\psi(x)$}
        \put(-38,50){$\psi(y)$}
        \put(-88,2){$v$}
        \caption{An illustration of the choice of the minimal cutset in the proof of Proposition~\ref{prop:qi}. In order to keep both parts, we do try to separate $\psi(x)$ from $\psi(y)$ but from $v$.}
\end{figure}

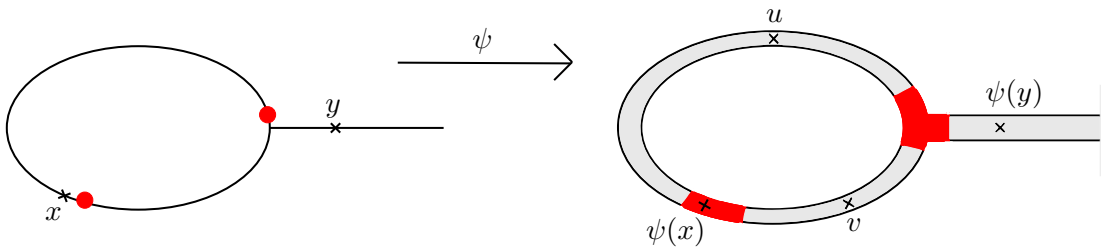
\begin{figure}[h!]
        \centering
        \includegraphics[width=0.9\linewidth]{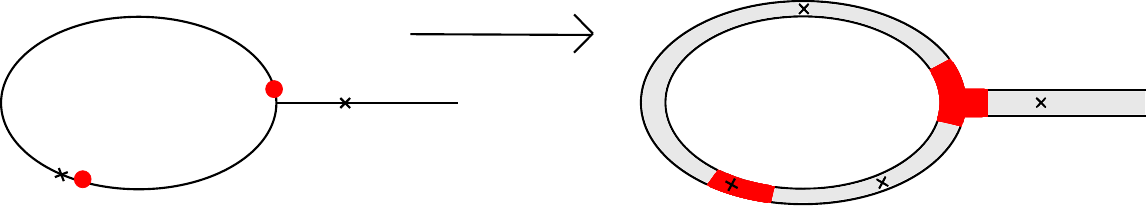}
        \put(-395,2){$x$}
        \put(-290,43){$y$}
        \put(-125,77){$u$}
        \put(-170,-5){$\psi(x)$}
        \put(-43,47){$\psi(y)$}
        \put(-95,-2){$v$}
        \put(-235,67){$\psi$}
        \caption{An example where we need to replace both vertices $\psi(x)$ and $\psi(y)$, and instead consider a minimal cutset with respect to two new vertices $u$ and $v$.}
\end{figure}


\begin{proof}


Let $A$ and $B$ be two constants. Let also $\scr G$ and $\scr H$ be two graphs as in the statement. Consider maps $\varphi:V(\mathscr{G})\to V(\mathscr{H})$ and $\psi:V(\mathscr{H})\to V(\mathscr{G})$ that are $A$-quasi-isometries and $A$-quasi-inverse of one another. Pick $\fatty\ge A^2+2A$, $\fat\ge A(A+2\fatty+B+3)$ and $C=2\fat+1$.  Let $\Pi$ be a finite minimal cutset between two vertices $x$ and $y$ in $\scr H$. Our goal is to prove that $\Pi$ is $C$-coarse connected.

Proceeding by contradiction, we assume that $\Pi$ is not $C$-coarse connected. Let us denote by $K_1$ a $C$-coarse connected component of $\Pi$, and by $K_2$ the union of all other $C$-coarse connected components of $\Pi$. By construction, and since we assume $\Pi$ not to be $C$-coarse connected, $(K_1,K_2)$ is a partition of $\Pi$ into two non-empty subsets, with $K_1$ being $C$-coarse connected.

As $\Pi$ is a cutset between $x$ and $y$, its complement consists of at least two connected components: one contains $x$, which we denote by $H_x$, and another one contains $y$, which we denote by $H_y$. Given any two distinct vertices $z$ and $z'$ in $\Pi$, there is a path from $z$ to $z'$ visiting only vertices of $H_x\cup\{z,z'\}$ --- it suffices to concatenate a path from $z$ to $x$ in $H_x\cup\{z\}$ with a path from $z'$ to $x$ in $H_x\cup\{z'\}$. Similarly, there is a path from $z$ to $z'$ visiting only vertices of $H_y\cup\{z,z'\}$.

Pick $z_1$ in the non-empty set $K_1$ and $z_2$ in the non-empty set $K_2$. Write $K_i(\fat)$ for the $\fat$-neighbourhood of $K_i$ and observe that, as $C\ge 2\fat+1$, the sets $K_1(\fat)$ and $K_2(\fat)$ are at distance at least 2 from each other. By what has been established in the previous paragraph, we can pick a path $\kappa_x$ from $z_1$ to $z_2$ that stays in $H_x\cup\{z_1,z_2\}$. Let $\kappa'_x$ denote the last excursion of $\kappa_x$ outside $K_1(\fat)$: this is the last subpath of $\kappa_x$ starting at distance 1 from $K_1(\fat)$, staying outside $K_1(\fat)$, and ending at $z_2$. Call $x'$ the starting point of $\kappa'_x$. Observe that, as $x'$ belongs to $H_x$ and is at distance at least $\fat+1$ from both $K_1$ and $K_2$, we have $B_{\fat}(x')\subset H_x$. Likewise, we can pick a path $\kappa'_y$ starting at distance 1 from $K_1(\fat)$, staying outside $K_1(\fat)$, ending at $z_2$, and such that its starting point $y'$ satisfies $B_{\fat}(y')\subset H_y$. By concatenating $\kappa'_x$ with the reversal of $\kappa'_y$, we obtain a path from $x'$ to $y'$ that does not intersect $K_1(\fat)$.

Let us now prove that there is a path from $x'$ to $y'$ that avoids $K_2(\fat)$. By definition of $x'$, we can pick a neighbour $x''$ of $x'$ that belongs to $K_1(\fat)$. Likewise, pick a neighbour $y''$ of $y'$ that belongs to $K_1(\fat)$. Since $K_1$ is $C$-coarse connected and $C\le2 \fat+1$, the set $K_1(\fat)$ is connected. Moreover, as it is disjoint from $K_2(\fat)$, we can pick a path from $x''$ to $y''$ that avoids $K_2(\fat)$. Appending $x'$ at the beginning and $y'$ at the end of this path yields a path from $x'$ to $y'$ that avoids $K_2(\fat)$.

For every $U\subset V_{\scr H}$, denote by $\hat U$ the $\fatty$-neighbourhood of $\psi(U)$ in $\scr G$. We first prove the following claim.

\begin{claim}
    \label{claim:absurd}
    The set $\hat\Pi$ is a finite cutset between $\psi(x')$ and $\psi(y')$. (Recall that this claim lies within a proof by contradiction.)
\end{claim}

\begin{proof}
Recall that $B_{\fat}(x')\subset H_x$ and that $\fat \ge A$, so that we have $B_A(x')\subset H_x$. Similarly, we have $B_A(y')\subset H_y$.

Let $\kappa=(\kappa_0,\dots,\kappa_\ell)$ denote a path from $\psi(x')$ to $\psi(y')$ in $\scr G$. We want to prove that $\kappa$ intersects $\hat \Pi$. Let us consider the sequence of vertices $(\varphi(\kappa_0),\dots,\varphi(\kappa_\ell))$. This sequence of vertices is generally not a path in $\scr H$ but any two subsequent terms of this sequence lie at distance at most $2A$. The vertex $\varphi(\kappa_0)=\varphi\circ \psi (x')$ belongs to $B_A(x')$, thus to $H_x$. Likewise, $\varphi(\kappa_\ell)$ belongs to $H_y$, hence does not belong to $H_x$. Pick the largest possible $i$ such that $\varphi(\kappa_i)\in H_x$. Note that such an $i$ exists and satisfies $i<\ell$. By construction, $\varphi(\kappa_i)\in H_x$, $\varphi(\kappa_{i+1})\notin H_x$ and $d_{\scr H}(\varphi(\kappa_i),\varphi(\kappa_{i+1}))\le 2 A$. Therefore, at least one of the two vertices $\varphi(\kappa_i)$ and $\varphi(\kappa_{i+1})$ must be at distance at most $A$ from $\Pi$. Denote by $j$ such an index in $\{i,i+1\}$. Then, we have $d_{\scr G}(\kappa_j, \psi\circ \varphi(\kappa_j))\le A$ and $d_{\scr G}(\psi( \varphi(\kappa_j)),\psi(\Pi))\le A^2+A$. Therefore, the triangle inequality yields $d_{\scr G}(\kappa_j,\psi(\Pi))\le A^2+2A\le\fatty$, thus proving that $\kappa_j$ belongs to $\hat \Pi$.   
\end{proof}

Now set $u=\psi(x')$ and $v=\psi(y')$. Our aim now is to show that neither $\hat{K_1}$ nor $\hat{K_2}$ is a cutset between $u$ and $v$. Let $j\in \{1,2\}$. Recall that there exists a path from $x'$ to $y'$ staying at distance strictly more than $\fat$ from $K_j$. Applying $\psi$ to this path yields a sequence of vertices at distance at most $2A$ from one another, starting at $u$ and ending at $v$. Arguing as in the proof of the claim and observing that $\fat\ge A^2+A\fatty+A$, we get the existence of a path from $u$ to $v$ avoiding $\hat{K_j}$, thus establishing that neither $\hat{K_1}$ nor $\hat{K_2}$ is a cutset between $u$ and $v$ in $\scr G$.

We can now conclude as follows. The set $\hat \Pi=\hat{K_1}\cup \hat{K_2}$ is a cutset between $u$ and $v$ and none of $\hat{K_1}$ and $\hat{K_2}$ is a cutset between $u$ and $v$. As a consequence, there is a minimal cutset contained in $\hat \Pi$ that intersects both $\hat{K_1}$ and $\hat{K_2}$. Therefore, we can pick $w_1\in \hat{K_1}$ and $w_2\in \hat{K_2}$ such that $d_{\scr G}(w_1,w_2)\le B$. By definition of $\hat{K_1}$ and $\hat{K_2}$, we can further pick $z'_1\in K_1$ and $z'_2\in K_2$ such that $d_{\scr G}(w_1,\psi(z'_1))\le \fatty$ and $d_{\scr G}(w_2,\psi(z'_2))\le \fatty$. Using the triangle inequality, we get $d_{\scr G}(\psi(z'_1),\psi(z'_2))\le 2\fatty+B$. Applying $\varphi$ and recalling that $\varphi\circ\psi$ is $A$-close to the identity, we get $d_{\scr H}(z'_1,z'_2)\le (2\fatty+B)A+3A\le C$. This contradicts the fact that $K_2$ is disjoint from the $C$-coarse connected component $K_1$ of $\Pi$, thus concluding the proof.
\end{proof}


\subsection{Proof of Theorem~\ref{thm:loc-bded}}
\label{sec:loc-bded}

    Let $(\mathscr{G}_n)_{n\geq 1}$ be a sequence of elements of $\mathfrak{G}$ converging to some $\mathscr{G}_\infty \in\mathfrak{G}$. We will prove that the sequence $(\weak(\mathscr{G}_n))_{n\geq 1}$ is bounded, which suffices to conclude that Theorem~\ref{thm:loc-bded} holds. 
    
    Pick an arbitrary root $o_n$ in every $\mathscr{G}_n$. By \cite[Theorem 2.3]{tesseratointon}, we can find a constant $A\ge 1$ and, for every $n\geq 1$, a normal subgroup $\Gamma_n$ of $\mathrm{Aut}(\mathscr{G}_n)$ such that, for every $n\geq 1$, the following conditions hold:
    \begin{enumerate}
    \item \label{item:fibre} every fibre of the projection $V(\mathscr{G}_n)\to V(\mathscr{G}_n)/\Gamma_n$ has diameter at most $A$,
    \item the group $\mathsf{Aut}(\mathscr{G}_n)/\Gamma_n$ contains a nilpotent subgroup $G_n$ of rank, depth and index at most $A$,
    \item the set $S_n=\{g\in \mathsf{Aut}(\mathscr{G}_n)/\Gamma_n\,:\, d_{\mathscr{G}_n/\Gamma_n}(g\cdot o_n,o_n)\le 1\}$ is a finite generating subset of $\mathsf{Aut}(\mathscr{G}_n)/\Gamma_n$,
    \item \label{item:stab} every vertex stabiliser of the action of $\mathsf{Aut}(\mathscr{G}_n)/\Gamma_n$ on $\mathscr{G}_n/\Gamma_n$ has cardinality at most $A$,
    \item \label{item:qi} the graph $\mathscr{G}_n$ is $A$-quasi-isometric to the Cayley graph $\mathscr{H}_n$ of $\mathsf{Aut}(\mathscr{G}_n)/\Gamma_n$ with respect to $S_n$.
\end{enumerate}

Our aim now is to show that the cardinality of $S_n$ is uniformly bounded by a constant independent of $n$. As the sequence $(\mathscr{G}_n)_{n\geq 1}$ is convergent, it has bounded degree. Therefore, in item~\ref{item:fibre}, there is a uniform bound not only on the diameter but also on the cardinality of the fibres. Since the sequence $(\mathscr{G}_n)_{n\geq 1}$ has bounded degree, this implies that the sequence $(\mathscr{G}_n/\Gamma_n)_{n\geq 1}$ has bounded degree as well. Combining this with item~\ref{item:stab}, we deduce that the cardinality of $S_n$ admits a uniform upper bound independent of $n$.


As the index of $G_n$ in $\mathsf{Aut}(\mathscr{G}_n)/\Gamma_n$ is bounded, we can find a constant $B\ge 2$ such that, for every $n\geq 1$, there is a Cayley graph $\mathscr{H}'_n$ of $G_n$ that is $B$-quasi-isometric to $\mathscr{H}_n$. Since the sequence $(\mathscr{H}_n)_{n\geq 1}$ has bounded degree and the index of $G_n$ in $\mathsf{Aut}(\mathscr{G}_n)/\Gamma_n$ is bounded, we can further assume, up to readjusting the value of $B$, that the degree of $\mathscr{H}'_n$ is bounded by $B$. This is a quantitative version of the classical fact that a finite-index subgroup of a finitely generated group is finitely generated, and it can be proved by arguing as in the proofs of \cite[Theorem 8.37 and Corollary 8.47]{drutu-kapovitch}.

Combining item~\ref{item:qi} with the above, we deduce that there is a constant $C\ge 1$ such that, for every $n\geq 1$, the graph $\mathscr{G}_n$ is $C$-quasi-isometric to a Cayley graph $\mathscr{H}'_n$ of $G_n$ of degree at most $B$. Let $H_{A,B}$ denote the free nilpotent group with $B$ generators and of depth $A$. This is a nilpotent group with the following universal property: for any nilpotent group $K$ generated by (at most) $B$ elements and with depth at most $A$, there exists a group homomorphism from $H_{A,B}$ to $K$. As $G_n$ is nilpotent of depth at most $A$ and because its Cayley graph $\mathscr{H}'_n$ has degree at most $B$, the universal property implies that we can realise $\mathscr{H}'_n$ as a quotient of the Cayley graph $\mathscr{U}$ of $H_{A,B}$ with respect to its standard generating set. This means in particular that there exists a surjective group homomorphism from $H_{A,B}$ to $G_n$ that maps the standard generating set of $H_{A,B}$ to the generating set of $G_n$ that defines $\mathscr{H}'_n$.

The Cayley graph $\mathscr{U}$ of $H_{A,B}$ has at least quadratic growth, as we have taken its number of generators to be $B\ge 2$. Since $\mathscr{U}$ is a Cayley graph of a nilpotent group, this implies that it is one-ended. Besides, the group $H_{A,B}$ is finitely generated and nilpotent, hence admits a finite presentation. By  \cite{babson1999cut, timarcutset}, the previous two sentences imply that $\pei_E(\mathscr{U})<\infty$. Applying Corollary~\ref{coro:mon} and Lemma~\ref{lem:bond}, we get that the sequence $(\weak(\mathscr{H}'_n))_{n\geq 1}$ is bounded. Recalling that for every $n$, the graph $\mathscr{G}_n$ is $C$-quasi-isometric to $\mathscr{H}'_n$, boundedness of $(\weak(\mathscr{G}_n))_{n\geq 1}$ now follows from Proposition~\ref{prop:qi}.\qed

\hypertarget{target:analytic}{\section{Proof of Theorems~\ref{thm:analyticity-poly} and \ref{thm:analyticity-local}}}
\label{sec:analytic}

We start by giving some necessary definitions. Let $\mathscr{G}=(V,E)$ be a graph in $\mathfrak{G}$. For $n\ge m\ge 1$ and $x\in V$, we define the uniqueness event 
$$U(m,n,x)=\{\text{there exists at most one cluster in $B_n(x)$ intersecting $B_m(x)$ and $S_n(x)$}\},$$
where $S_n(x)$ denotes the \emph{sphere} of radius $n$ around $x$, i.e.\ the set of vertices at distance $n$ from $x$.
In \cite{supercritpoly}, it is proved that for every $p>p_c$, every $x\in V$, and every $n$ large enough, we have
\begin{equation}\label{eq: local uniqueness}
\mathbb{P}_p(B_{n}(x)\longleftrightarrow S_{10n}(x),~U(2n,5n,x))\geq 1-e^{-\sqrt{n}}.    
\end{equation}
This is the main ingredient we will use from that paper.

With this result at hand, our strategy for proving Theorems~\ref{thm:analyticity-poly} and \ref{thm:analyticity-local} follows a renormalisation approach. To this end, let us now introduce some events that are suitable for renormalisation purposes. Let $N\geq 1$ be a positive integer, and consider $x\in V$. A vertex $x$ is called \defini{$N$-good} if for every vertex $z$ that either coincides with $x$ or is a neighbour of $x$, the event $\{B_{N}(z)\longleftrightarrow S_{10N}(z),~ U(2N,5N,z)\}$ occurs. Otherwise, the vertex $x$ is called \defini{$N$-bad}.

Instead of working directly with the cluster of the origin $\mathcal{C}_o$, we shall work with a coarse-grained version of it. 
We write $\mathcal{C}_o(N)$ for the set of vertices $x$ such that $d_{\mathscr{G}}(x,\mathcal{C}_o)\leq N$. Besides, for any $A\subset V$, let $\partial A$ denote its \defini{exposed boundary}, i.e.\ the set of vertices in $A$ that are incident to the infinite component of $\mathscr{G}\setminus A$. We will be interested in $\partial \mathcal{C}_o(N)$, the exposed boundary of $\mathcal{C}_o(N)$.

For a subset $A$ of $V$, let us write $\textup{diam}(A)$ for its diameter, i.e.
\[\textup{diam}(A):=\sup\{d_{\mathscr{G}}(x,y) :\, x,y\in A\}.\]We first show that when $\mathcal{C}_o$ is of large but finite diameter, every point in $\partial \mathcal{C}_o(N)$ is $N$-bad.

\begin{proposition}\label{prop: bad vert}
Let $N\geq 1$ and let $\omega$ be a configuration such that $20N<\textup{diam}(\mathcal{C}_o)<\infty$. Then every vertex of $\partial \mathcal{C}_o(N)$ is $N$-bad.    
\end{proposition}

\begin{proof}
Consider a vertex $x \in \partial \mathcal{C}_o(N)$ and let $y \in V \setminus \mathcal{C}_o(N)$ be a neighbour of $x$ (such a $y$ exists). If there is no open path connecting $B_N(y)$ to $S_{10N}(y)$, then $x$ is $N$-bad, and there is nothing to prove. Thus, we may assume that there exists an open path $\gamma_1$ connecting $B_N(y)$ to $S_{10N}(y)$.  
We now show that the event $U(2N,5N,x)$ does not occur. Indeed, since $B_N(y) \subset B_{2N}(x)$ and $S_{5N}(x) \subset B_{10N}(y)$, the restriction of $\gamma_1$ to $B_{5N}(x)$ forms an open path connecting $B_{2N}(x)$ to $S_{5N}(x)$.  
Moreover, the intersection of $\mathcal{C}_o$ with $B_{10N}(x)$ contains a path $\gamma_2$ that connects $B_N(x)$ to $S_{10N}(x)$. This follows from the fact that $\mathcal{C}_o$ contains a vertex in $B_N(x)$ by definition, and that $\textup{diam}(\mathcal{C}_o) > 20N$.  
Since $y \in V \setminus \mathcal{C}_o(N)$, the paths $\gamma_1$ and $\gamma_2$ must belong to distinct clusters. This implies that $U(2N,5N,x)$ does not occur, and hence $x$ is $N$-bad, as required.  
\end{proof}

Let us fix a constant $t$ such that the conclusion of Theorem~\ref{thm:loc-bded} holds. In particular, $\partial \mathcal{C}_o(N)$ is $t$-coarse connected for every $N\geq 1$. Below, we will take $N$ to be sufficiently large, and in particular choose $N\geq t$. 

\begin{definition}
Let $N\geq t$, and let $\omega$ be a configuration such that $20N<\mathrm{diam}(\mathcal{C}_o)<\infty$. Define $\mathcal{I}$ to be the $t$-coarse connected component of $\partial \mathcal{C}_o(N)$ within the set of $N$-bad vertices in $\omega$. Define $\overline{\mathcal{I}}$ to be the set of vertices $x\in V$ such that $d_{\mathscr{G}}(x,\mathcal{I})\leq 10N$.    
\end{definition}

Let us now discuss the motivation for defining $\mathcal{I}$ in this way. Our goal is to express $\theta$ in terms of events $A_n$ such that each $A_n$ is measurable with respect to the state of $O(n)$ edges, and $\mathbb{P}_p(A_n)$ decays exponentially in $n$. This exponential decay is crucial for obtaining an analytic extension of $\theta$ to a complex neighbourhood of $p$. While the set $\partial \mathcal{C}_o(N)$ has an exponential tail, since it consists of $N$-bad vertices, it does not satisfy the measurability condition: the state of the edges in a small neighbourhood of $\partial \mathcal{C}_o(N)$ does not necessarily determine the event $\{o\centernot\longleftrightarrow \infty\}$. This arises from the possibility that the infinite cluster -- and in particular its boundary -- comes close to $\mathcal{C}_o$. However, this event gives rise to additional $N$-bad vertices, and by the $t$-connectivity of boundaries, these vertices form a $t$-coarse connected set consisting of $N$-bad vertices that contains $\partial \mathcal{C}_o(N)$. The definition of $\mathcal{I}$ is designed precisely to capture this phenomenon.

The following lemma will allow us to obtain an exact expression for $\theta$, and is the main reason why we defined $\overline{\mathcal{I}}$ as above.

\begin{lemma}\label{lem:cutset}
Let $N\geq t$. Consider a configuration $\omega$ such that $20N<\textup{diam}(\mathcal{C}_o)<\infty$ and $\mathcal{I}$ is finite. Then the closed edges lying in $\overline{\mathcal{I}}$ form a bond-cutset between $o$ and $\infty$.
\end{lemma}

\begin{proof}
Let us fix a configuration $\omega$ such that $20N<\textup{diam}(\mathcal{C}_o)<\infty$, and let us proceed by contradiction. We can thus choose some path $\gamma$ from $o$ to $\infty$ such that all of its edges lying in $\overline{\mathcal{I}}$ are open. 

Now let $u$ be the last vertex of $\mathcal{C}_o$ that $\gamma$ visits. Then $u=\gamma(k)$ for some $k$, and $\gamma(k+1)\not \in \mathcal{C}_o$, hence the edge $\{u,\gamma(k+1)\}$ cannot be open. In particular, this implies that $u\not \in \mathcal{I}$. Notice that $u$ needs to be in a finite connected component $B$ of $\mathscr{G}\setminus\partial \mathcal{C}_o(N)$. As $u$ does not belong to $\mathcal{I}$ and because, $\mathscr{G}\setminus \mathcal{I}\subset \mathscr{G}\setminus\partial \mathcal{C}_o(N)$, it is the case that $u$ lies in some finite connected component $A$ of $\mathscr{G}\setminus \mathcal{I}$. We emphasise that all edges of $\mathscr{G}$ with both endpoints out of $\mathcal{I}$ are included in the definition of $A$, irrespective of whether they are open or closed. Notice that all vertices of $\partial A$ are $N$-good. Now let $v$ be the first vertex of $\partial A$ after $u$ that $\gamma$ visits; such a vertex exists as $\gamma$ is infinite and $v$ belongs to $A$. We claim that there is an open path from $u$ to $v$. This implies that $v$ lies in $\mathcal{C}_o$, which contradicts the definition of $u$, and thus the desired result follows.

Let us prove the claim. Note that all edges of $\gamma$ in $B_{5N}(v)$ are open, since $B_{5N}(v)\subset\overline{\mathcal{I}}$ and all edges of $\gamma$ in $\overline{\mathcal{I}}$ are open. If the restriction of $\gamma$ between $u$ and $v$ lies in $B_{5N}(v)$, then this restriction gives the desired open path. Thus, we may assume until the end of the proof that this is not the case, which implies the following property:
\begin{equation}
\label{eq:connect}
    \text{the path $\gamma$ connects $v$ to $S_{5N}(v)$ by an open path.}
\end{equation}We want to argue that there exist vertices $w\in\mathcal{C}_o$ and $w'\in \partial A$ such that $d_{\mathscr{G}}(w,w')\leq N$. To this end, introduce $\tilde A$ (resp.\ $\tilde B$) the complement of the infinite connected component of the complement of $A$ (resp.\ $B$). If $\mathcal{C}_o$ is not a subset of $\tilde A$, then it intersects the inner vertex-boundary\footnote{The inner vertex-boundary of a set $A$ is the set of all vertices of $A$ that are adjacent to some vertex out of $A$.} of $\tilde A$, which is $\partial A$. In this case, we may even pick $w=w'$. On the opposite, if $\mathcal{C}_o$ is a subset of $\tilde A$, then it is a subset of $\tilde B$ lying at distance $N$ from $\partial \mathcal{C}_o(N)$, which is in this case equal to $\partial_V \tilde B$. In this second case, existence of a suitable pair $(w,w')$ follows from the fact that $\mathcal{C}_o\subset \tilde A\subset \tilde B$, so that any path from $\mathcal{C}_o$ to $\partial_V \tilde B$ has to cross the inner vertex-boundary of $\tilde A$.

Now we connect $u$ to $v$ by an open path as follows. Since $\weak(\mathscr{G})\leq t$, we know that $\partial A$ is $t$-coarse connected, which means that we can find a sequence of vertices $v_1=w',\ldots,v_{\ell}=v$ in $\partial A$ such that $d_{\mathscr{G}}(v_i,v_{i+1})\leq t$ for every $i<\ell$. As all vertices $v_1,\ldots,v_{\ell}$ are $N$-good, $B_N(v_i)$ is connected to $S_{10N}(v_i)$ by an open for every $i\le\ell$. Furthermore, $w$ is connected to $S_{5N}(w')$ by an open path due to the fact that $\mathrm{diam}(\mathcal{C}_w)>20N$, and by \eqref{eq:connect}, the vertex $v=v_{\ell}$ is connected to $S_{5N}(v_{\ell})$ by an open path. Using that $N\geq t$ and that all the events $U(2N,5N,v_i)$ happen, we deduce that we can connect $w$ to $v_{\ell}=v$ by an open path. Since $w\in \mathcal{C}_o$, we can thus connect $u$ to $v$ by an open path, which proves the claim and completes the proof.
\end{proof}

With the above result at hand, we now give the following definition.

\begin{definition}
\label{defi:interface}
Let $\mathscr{I}\subset V$ be a finite set of vertices. We say that $\mathscr{I}$ is an \defini{interface} if it is $t$-coarse connected and either $o\in \mathscr{I}$ or $o$ lies in a finite connected components of $\mathscr{G}\setminus \mathscr{I}$. Let $\overline{\mathscr{I}}$ denote the set of vertices $x$ such that $d_{\mathscr{G}}(x,\mathscr{I})\leq 10N$. We say that an interface $\mathscr{I}$ \defini{occurs in $\omega$ at scale $N$} if
\begin{enumerate}
    \item\label{bad} every $x\in \mathscr{I}$ is bad,
    \item\label{good} every $x\in V\setminus \mathscr{I}$ such that $d_{\mathscr{G}}(x,\mathscr{I})\leq t$ is good,
    \item\label{finite} the closed edges of $\overline{\mathscr{I}}$ form a bond-cutset between $o$ and $\infty$.
\end{enumerate}
When clear from context, we simply write that $\mathscr{I}$ occurs.
\end{definition}

The following result is a direct consequence of the above definition.

\begin{lemma}\label{lem: disjoint}
Given a configuration $\omega$, the family of all occurring interfaces consists in mutually disjoint sets.
\end{lemma}
\begin{proof}
Since interfaces are $t$-coarse connected and because of items~\ref{bad} and \ref{good} in Definition~\ref{defi:interface}, occurring interfaces are in particular $t$-coarse connected components of the set of bad vertices. The desired result follows.
\end{proof}

Let us write $\mathcal{C}_x^{\mathrm{bad}}(N)$ for the $t$-coarse connected component of $x$ in $\omega$ within the set of $N$-bad vertices, with the convention that $\mathcal{C}_x^{\mathrm{bad}}(N)=\varnothing$ if $x$ is $N$-good. We now show that for every $p>p_c$, for every $N$ large enough (depending on $p$), $\mathcal{C}_x^{\mathrm{bad}}(N)$ is almost surely finite, which will allow us to deduce that $\overline{\mathcal{I}}$ is almost surely finite. The proof is similar to that of \cite[Lemma 10.1]{supercritpoly}. We include a proof for completeness and to justify the uniformity in $p$ in the statement. The latter is crucial for obtaining a suitable expression for $\theta$ that holds in an open neighbourhood. 

\begin{proposition}\label{prop: finite}
Let $\mathscr{G}\in \mathfrak G$ and consider $p_0\in (p_c(\mathscr{G}),1]$. There exist $N_0\geq 1$ and $c>0$ such that the following holds. For every $N\geq N_0$, there exists $\varepsilon>0$ such that
\[
\forall\; p\in (p_0-\varepsilon,p_0+\varepsilon),\quad \forall\; n\geq 1, \qquad \mathbb{P}_p(|\mathcal{C}_x^{\mathrm{bad}}(N)|\geq n)\leq e^{-cn}.
\]
In particular, $\mathcal{C}_x^{\mathrm{bad}}(N)$ is finite almost surely.
\end{proposition}

\begin{proof}
We will first prove the desired exponential decay for every fixed parameter $p>p_c(\mathscr{G})$, and then use a continuity argument to deduce the uniformity of the constant $c$ in terms of $p$.

Let us consider some $N>0$ that will be chosen along the way to be large enough depending on $p$.  
We will estimate $\mathbb{P}_p(|\mathcal{C}_x^{\mathrm{bad}}(N)|\geq n)$ using a union bound. To this end, let $\mathscr{I}$ be a $t$-coarse connected set of vertices that contains $x$ such that $|\mathscr{I}|=n$. Write $r=10N+1$.
We say that a set of vertices $A$ is $2r$-separated if $d_{\mathscr{G}}(u,v)>2r$ for all distinct $u,v \in A$. Let $\mathscr{I}'\subset \mathscr{I}$ be a $2r$-separated set that is maximal for inclusion among $2r$-separated subsets of $\mathscr{I}$. Notice that by maximality, for every $y\in \mathscr{I}$, there exists $z\in \mathscr{I}'$ such that $y\in B_{2r}(z)$. Hence, $|\mathscr{I}'|\geq |\mathscr{I}|/|B_{2r}|$.

We claim that $\mathscr{I}'$ is $(4r+t)$-coarse connected. Indeed, consider two vertices $u,v \in \mathscr{I}'$ and let $z,w\in \mathscr{I}$ be two vertices such that $d_{\mathscr{G}}(u,z)\leq 2r$ and $d_{\mathscr{G}}(v,w)\leq 2r$. As $\mathscr{I}$ is $t$-coarse connected, we can pick a $t$-coarse connected path $\gamma$ in $\mathscr{I}$ from $z$ to $w$, i.e.\ consecutive vertices in $\gamma$ are at distance at most $t$ apart in $\mathscr{G}$.  Consider two consecutive vertices $a,b$ in $\gamma$, and pick two vertices $a', b'\in \mathscr{I}'$ such that $d_{\mathscr{G}}(a,a')\leq 2r$ and $d_{\mathscr{G}}(b,b')\leq 2r$. Such vertices exist by the maximality of $\mathscr{I}'$. We then have $d_{\mathscr{G}}(a',b')\leq 4r+t$. Iterating this procedure we can construct a $(4r+t)$-coarse connected path from $u$ to $v$. This proves the claim.

We can thus conclude that when the event $\{|\mathcal{C}_x^{\mathrm{bad}}(N)|\geq n\}$ happens, there exists a $(4r+t)$-coarse connected set of vertices $\mathcal{S}$ that is $2r$-separated, has cardinality $|\mathcal{S}|\geq n/|B_{2r}|$, contains a vertex in $B_{2r}(x)$, and such that every vertex in $\mathcal{S}$ is $N$-bad. The union bound gives
$$
\mathbb{P}_p(|\mathcal{C}_x^{\mathrm{bad}}(N)|\geq n)\leq \sum_{m\geq n/|B_{2r}|}\sum_{\mathcal{S}\in P_m}\mathbb{P}_p(\forall\,~y\in \mathcal{S}, y \text{ is N-bad}),
$$
where $P_m$ is the set of $\mathcal{S}$ as above that have cardinality $m$. Note that $\mathcal{S}$ being $2r$-separated implies that the events $\{y \text{ is N-bad}\}$ are independent from each other when $y$ ranges over $\mathcal{S}$, hence 
$$
\mathbb{P}_p(\forall\, y\in \mathcal{S},~y \text{ is N-bad})=\mathbb{P}_p(o \text{ is N-bad})^m.
$$
A union bound and \eqref{eq: local uniqueness} give $\mathbb{P}_p(o \text{ is $N$-bad})\leq (D+1)e^{-\sqrt{N}}$ provided that $N$ is large enough, where $D$ is the degree of $\mathscr{G}$.

It remains to estimate the cardinality of $P_m$. Consider the graph $\mathscr{G}^{(i)}$ with vertex set $V$, where we draw an edge between $u$ and $v$ if $1\le d_{\mathscr{G}}(u,v)\leq i$. Then each $\mathcal{S}\in P_m$ is connected in $\mathscr{G}^{(4r+t)}$. Since each vertex in $\mathscr{G}^{(4r+t)}$ has degree at most $|B_{4r+t}|-1$, it follows from \cite[Lemma 5.1]{KestenBook} that 
$$|P_m|\leq |B_{2r}| |B_{4r+t}|^m e^m.$$
Now recall that $r=10N+1$. Since $|B_{4r+t}|$ grows polynomially fast, we can choose $N$ to be large enough so that 
\begin{equation}\label{eq:choose N}
    |B_{4r+t}|e (D+1)e^{-\sqrt{N}} <1.
\end{equation}
With this choice, we see that $\mathbb{P}_p(|\mathcal{C}_x^{\mathrm{bad}}(N)|\geq n)$ decays exponentially fast. The uniformity in $p$ now follows from the continuity of the function $p\mapsto \mathbb{P}_p(o \text{ is $N$-bad})$.
\end{proof}

As a corollary we obtain the following result regarding finiteness of $\mathscr{I}$.

\begin{corollary}\label{cor: finite}
Let $\mathscr{G}\in \mathfrak G$ and consider $p_0\in (p_c(\mathscr{G}),1]$. There exists $N_0\geq 1$ such that the following holds. For every $N\geq N_0$, there exists $\varepsilon>0$ such that
\[
\forall\; p\in (p_0-\varepsilon,p_0+\varepsilon), \qquad \mathbb{P}_p(|\mathcal{I}|<\infty \mid 20N<\mathrm{diam}(\mathcal{C}_o)<\infty)=1.
\]
\end{corollary}
\begin{proof}
This follows from Propositions~\ref{prop: bad vert} and \ref{prop: finite}.
\end{proof}

It follows from Corollary~\ref{cor: finite}, and Lemma~\ref{lem:cutset} that for every $p_0>p_c$ and for every $N$ large enough, whenever $20N<\textup{diam}(\mathcal{C}_o)<\infty$, an interface occurs at scale $N$. Moreover, whenever an interface occurs, $o$ is not connected to infinity by property \ref{finite} of Definition~\ref{defi:interface}. Thus, for some $\varepsilon>0$ we have the exact equality
\begin{equation}\label{eq: theta equality}
1-\theta(p)=\mathbb{P}_p(\mathrm{diam}(\mathcal{C}_o)\leq 20N)+\mathbb{P}_p(\mathrm{diam}(\mathcal{C}_o)> 20N,\text{ an interface occurs at scale } N)
\end{equation}
for every $p\in (p_0-\varepsilon,p_0+\varepsilon)$,
and we emphasise that this expression holds with the same constant $N$ in the whole interval $(p_0-\varepsilon,p_0+\varepsilon)$.

Let us now fix an $N$ such that \eqref{eq: theta equality} holds. Note that $\mathbb{P}_p(\mathrm{diam}(\mathcal{C}_o)\leq 20N)$ is trivially analytic since it is a polynomial in $p$. To handle the second term, we wish to use the inclusion-exclusion principle to express $\mathbb{P}_q(\mathrm{diam}(\mathcal{C}_o)> 20N,\text{ an interface occurs at scale }N)$ as a sum 
over \defini{multi-interfaces}, i.e.\ finite non-empty collections of \defini{disjoint} interfaces. The fact that we can restrict to collections of disjoint interfaces follows from Lemma~\ref{lem: disjoint}. We say that a multi-interface $\mathscr{M}=\{\mathscr{I}_1,\mathscr{I}_2,\ldots,\mathscr{I}_k\}$ \defini{occurs in $\omega$ at scale $N$} if all $\mathscr{I}_i$ occur, and we write $|\mathscr{M}|:=\sum_{i=1}^k |\mathscr{I}_i|$ for its size. 

Some care is required to show that the inclusion-exclusion expansion converges. This follows from the next result, where we will show that for every $N$ large enough, the contribution of multi-interfaces of size $n$ decays exponentially fast in $n$. We stress that the decay being exponential is crucial for obtaining the convergence of the inclusion-exclusion expansion for complex values of the parameter $p$.

\begin{proposition} \label{prop: exp dec}
Let $\mathscr{G}$ be a graph in $\mathfrak{G}$, and let $p_0\in (p_c(\mathscr{G}),1]$. There exist constants $\varepsilon$, $r$, $N, c, C>0$ such that
$$\sum_{\substack{\mathscr{M}: |\mathscr{M}|=n}}\mathbb{P}_{p,\mathscr{H}}(\mathscr{M} \text{ occurs at scale }N)\leq Ce^{-cn}$$
for every $\mathscr{H}\in \mathfrak{G}$ such that $R(\mathscr{H},\mathscr{G})\geq r$, every $p\in (p_0-\varepsilon,p_0+\varepsilon)\cap[0,1]$ and every $n\geq 1$.
\end{proposition}

In order to prove the above proposition, we first need some intermediate results. Let us fix a bi-infinite geodesic $\gamma=(\ldots,\gamma(-1),\gamma(0),\gamma(1),\ldots)$ in $\mathscr{G}$ such that $\gamma(0)=o$. The existence of $\gamma$ classically follows from considering a sequence of geodesics of length $2i$ starting at $o$, applying an automorphism to map the middle vertex of each geodesic to $o$, and then using a compactness argument. 

\begin{lemma}\label{lem:geodesic}
Let $\mathscr{I}$ be an interface such that $|\mathscr{I}|=n$. Then there exists $i$ such that $\gamma(i)\in \mathscr{I}$ and $0\leq i\leq t(n-1)$.
\end{lemma}
\begin{proof}
Recall that either $\mathscr{I}$ contains $o$ or $o$ lies in a finite component of its complement in $\mathscr{G}$. This implies that $\mathscr{I}$ intersects both infinite geodesic rays  $(\gamma(0),\gamma(1),\ldots)$ and $(\ldots,\gamma(-1),\gamma(0))$ at some vertices $x^+$ and $x^-$, respectively. Since $\mathscr{I}$ is $t$-coarse connected and $|\mathscr{I}|=n$, it follows that $d_{\mathscr{G}}(x^+,x^-)\leq tn$, which implies that $x^{+}=\gamma(i)$ for some $0\leq i\leq t(n-1)$.
\end{proof}

Our strategy for proving Proposition~\ref{prop: exp dec} is based on the observation that 
$$\sum_{\mathscr{M}: |\mathscr{M}|=n}\mathbb{P}_p(\mathscr{M} \text{ occurs at scale }N)=\mathbb{E}_p(\mathcal{N}_n),$$
where $\mathcal{N}_n$ is the (random) number of occurring multi-interfaces of size $n$. Using combinatorial arguments, we now obtain a deterministic sub-exponential upper bound on $\mathcal{N}_n$ which holds uniformly over all configurations $\omega$.

\begin{lemma}\label{lem:Hardy-Ram}
There is a constant $c>0$ such that for every $\mathscr{G}\in \mathfrak{G}$ and any configuration $\omega$ we have
$$\mathcal{N}_n\leq e^{c\sqrt{tn}}$$
for every $n\geq 1$.
\end{lemma}
\begin{proof}
Let $\mathscr{M}=\{\mathscr{I}_1,\mathscr{I}_2,\ldots,\mathscr{I}_k\}$ be a multi-interface that occurs in $\omega$ at scale $N$ and has size $n$. By Lemma~\ref{lem:geodesic}, there exist integers $i_1,i_2,\ldots,i_k$ such that $\gamma(i_j)\in \mathscr{I}_j$ and $0\leq i_j\leq t(|\mathscr{I}_j|-1)$. Since occurring interfaces are disjoint, it follows that the numbers $i_1,i_2,\ldots,i_k$ are all distinct. Hence, the set
$\{i_1,i_2,\ldots,i_k\}\setminus\{0\}$ forms a partition of an integer $r\leq tn$. Using a well-known result of Hardy and Ramanujan on the number of partitions \cite{HarRam}, we obtain that the number of all possible sets $\{i_1,i_2,\ldots,i_k\}$ produced in this way is at most $e^{c\sqrt{tn}}$ for some constant $c>0$. Since for each configuration $\omega$, the map 
$\{\mathscr{I}_1,\mathscr{I}_2,\ldots,\mathscr{I}_k\}\mapsto \{i_1,i_2,\ldots,i_k\}$ is injective, it follows that 
$$\mathcal{N}_n\leq e^{c\sqrt{tn}},$$
as desired.
\end{proof}

Using the above lemma, we now prove Proposition~\ref{prop: exp dec}.

\begin{proof}[Proof of Proposition~\ref{prop: exp dec}]
Consider $\mathscr{G}\in \mathfrak{G}$. By Theorem~\ref{thm:finitary-timar}, we can pick $t>0$ be such that $\weak(\mathscr{H})\leq t$ for every $\mathscr{H}$ in some neighbourhood of $\mathscr{G}$ within $\mathfrak{G}$. Let us also fix some $N\geq t$ that will be chosen along the way to be large enough depending only on $p$ and $\mathscr{G}$. We first prove the assertion for fixed $p>p_c(\mathscr{G})$ and for any $\mathscr{H}$ satisfying $R(\mathscr{H},\mathscr{G})\geq k$ for some $k$ to be determined. In particular, we let $k$ be large enough so that $\weak(\mathscr{H})\leq t$. The uniformity in $p$ will then follow by a continuity argument.

By Lemma~\ref{lem:Hardy-Ram}, we have 
\begin{equation}\label{eq:expectation-probability}
\mathbb{E}_{p,\mathscr{H}}(\mathcal{N}_n)\leq e^{c\sqrt{tn}}\mathbb{P}_{p,\mathscr{H}}(\mathcal{A}_n),    
\end{equation}
where $\mathcal{A}_n:=\{\text{there exists at least one multi-interface $\mathscr{M}$ of size $n$ that occurs}\}$, hence it suffices to show that the latter probability decays exponentially. To this end, we will argue as in the proof of Proposition~\ref{prop: finite}.

Let $r=10N+1$, and recall that $\mathscr{H}^{(i)}$ is the graph with vertex set $V$, where we draw an edge between $u$ and $v$ if $1\leq d_{\mathscr{H}}(u,v)\leq i$. Recall also that a set of vertices $A$ is $2r$-separated if $d_{\mathscr{G}}(u,v)>2r$ for all distinct $u,v \in A$. Given a multi-interface $\mathscr{M}=\{\mathscr{I}_1,\mathscr{I}_2,\ldots,\mathscr{I}_k\}$ of size $n$, let $S_1,S_2,\ldots,S_{\ell}$ be the connected components of $\bigcup_{j=1}^k\mathscr{I}_j$ in the graph $\mathscr{H}^{(2r)}$, i.e.\ each $S_j$ is $2r$-coarse connected, and is maximal for inclusion with respect to this property. Now, for each $S_j$, let $S'_j\subset S_j$ be a $2r$-separated set that is maximal for inclusion with respect to this property. Then $S':=\bigcup_{j=1}^{\ell} S'_j$ is $2r$-separated, and 
\begin{equation}\label{eq: size lower bound}
|S'|\geq \frac{|\mathscr{M}|}{|B_{2r,\mathscr{H}}|}=\frac{n}{|B_{2r,\mathscr{H}}|}.
\end{equation}
By arguing as in the proof of Proposition~\ref{prop: finite}, we see that each $S'_j$ is $6r$-coarse connected.

We now estimate the number of all possible $S'$. 
To this end, we first wish to bound $\ell$, which we recall is the number of components of $S'$. As $2r\geq t$, we have the inequality $\ell\leq k$. To estimate $k$, for each $i\le k$, let $m_i\ge 0$ be such that $\gamma(m_i)\in \mathscr{I}_i$ and $m_i\leq t(|\mathscr{I}_i|-1)$. Such $m_i$ indeed exist by Lemma~\ref{lem:geodesic}. Up to relabelling the interfaces of $\mathscr{M}$ we can assume that $m_1<m_2<\ldots<m_k$. Then $m_i\geq i-1$, which implies that $|\mathscr{I}_i|\geq i/t$. Hence 
\[
n=\sum_{i=1}^k |\mathscr{I}_i|\geq \sum_{i=1}^k \frac{i}{t}=\frac{k(k+1)}{2t}\geq \frac{k^2}{2t},
\]
which in turn implies that $\ell\leq k\leq \sqrt{2tn}$.

Using Lemma~\ref{lem:geodesic} again and the maximality in the definition of $S'_j$, we deduce that each $S'_j$ contains a vertex $x$ such that $d_{\mathscr{H}}(x,\gamma(i))\leq 2r$ for some $0\leq i\leq t(n-1)$. Hence there are at most $tn|B_{2r,\mathscr{H}}|$ choices for $x$. Combining the latter with \eqref{eq: size lower bound}, \cite[Lemma 5.1]{KestenBook}, and using that the number of partitions of $m$ is at most $e^{c\sqrt{m}}$, we obtain that there are at most
$$
\sum_{k=1}^{\sqrt{2tn}} (t n|B_{2r,\mathscr{H}}|)^k e^{c\sqrt{m}}|B_{6r,\mathscr{H}}|^m e^m\leq \sqrt{2tn}\, (t n|B_{2r,\mathscr{H}}|)^{\sqrt{2tn}} e^{c\sqrt{m}} |B_{6r,\mathscr{H}}|^m e^m
$$
possibilities for the number of all possible $S'$ of cardinality $m$. A union bound and independence now gives
\begin{equation}\label{eq:bounds 3}
\mathbb{P}_{p,\mathscr{H}}(\mathcal{A}_n)\leq \sqrt{2tn}\, (t n|B_{2r,\mathscr{H}}|)^{\sqrt{2tn}} \sum_{m\geq n|B_{2r,\mathscr{H}}|^{-1}}  |B_{6r,\mathscr{H}}|^m e^{c\sqrt{m}} e^m \, \mathbb{P}_{p,\mathscr{H}}(o \text{ is N-bad})^m.
\end{equation}
Writing $D$ for the (common) vertex-degree, we now choose $N$ to be large enough so that 
$$
|B_{6r,\mathscr{G}}| e \, \mathbb{P}_{p,\mathscr{G}}(o \text{ is N-bad})\leq |B_{6r,\mathscr{G}}| e (D+1) e^{-\sqrt{N}}:=\alpha<1,
$$
which is possible by the polynomial growth of $\mathscr{G}$ and \eqref{eq: local uniqueness}. We choose in turn $k$ to be large enough so that $k\geq 6r$. Then we have $\mathbb{P}_{p,\mathscr{H}}(o \text{ is N-bad})=\mathbb{P}_{p,\mathscr{G}}(o \text{ is N-bad})$ and
$$
|B_{6r,\mathscr{H}}| e \, \mathbb{P}_{p,\mathscr{H}}(o \text{ is N-bad})\leq \alpha.
$$
Thus, we see that the right-hand side of \eqref{eq:bounds 3} decays exponentially in $n$, uniformly over $\mathscr{H}$, as desired.

Uniformity in both $p$ and $\mathscr{H}$ now follows from continuity of the fucntion $p\mapsto\mathbb{P}_{p,\mathscr{G}}(o \text{ is N-bad})$. This completes the proof.
\end{proof}

We are now ready to prove Theorem~\ref{thm:analyticity-poly}. 

\begin{proof}[Proof of Theorem~\ref{thm:analyticity-poly}]
Let $p_0>p_c(\mathscr{G})$ and consider some $N$ and $\varepsilon$ as in Corollary~\ref{cor: finite} and Proposition~\ref{prop: exp dec}. Then we have $$1-\theta(p)=\mathbb{P}_p(D_N)+\mathbb{P}_p(D_N^c,\text{ an interface occurs at scale }N)$$
for every $p\in (p_0-\varepsilon,p_0+\varepsilon)\cap[0,1]$, where $D_N=\{\mathrm{diam}(\mathcal{C}_o)\leq 20N\}$. The function $p\mapsto\mathbb{P}_p(D_N)$ is a polynomial in $p$, hence analytic. For the second term, we use the inclusion-exclusion principle to obtain 
$$\mathbb{P}_p(D_N^c,\text{ an interface occurs at scale }N)=\sum_{\mathscr{M}=\{\mathscr{I}_1,\ldots, \mathscr{I}_k\}}(-1)^{k+1}\mathbb{P}_p(D_N^c, \text{ $\mathscr{M}$ occurs at scale }N).
$$
Each summand is a polynomial in $p$, which allows us to formally extend the parameter $p$ to the complex plane. We show that this series converges uniformly in a complex neighbourhood $\mathscr{O}$ of $p_0$, which classically implies that the series is an analytic function in the interior of $\mathscr{O}$. Then, as it coincides with $\mathbb{P}_p(D_N^c,\text{ an interface occurs at scale }N)$ on $\mathscr{O}\cap [0,1]$, it follows that $\mathbb{P}_p(D_N^c,\text{ an interface occurs at scale }N)$ admits an analytic extension on $\mathscr{O}$, hence so does $\theta$. 

To this end, we compare the value of the sum at a complex parameter $z$ with its value at a point close to $p_0$ by estimating the degree of each polynomial term. Note that the event $D_N^c$ depends only on the state of the edges with both endpoints in $B_{20N+1,\mathscr{G}}$, and the event $\{\mathscr{I}_i \textup{ occurs}\}$ depends only on the state of the edges with both endpoints in $\bigcup_{x\in \mathscr{I}_i} B_{20(t+1)N+1,\mathscr{G}}(x)$. Thus, for a multi-interface $\mathscr{M}=\{\mathscr{I}_1,\mathscr{I}_2,\ldots, \mathscr{I}_k\}$ of size $n$, the event $\{D_N^c,\text{ $\mathscr{M}$ occurs}\}$ depends on the state of at most $An$ edges for some constant $A>0$ that only depends on $t$ and the cardinality of $B_{20(t+1)N+1,\mathscr{G}}$. Let us introduce
\[
F_n(p):=\sum_{\substack{\mathscr{M}=\{\mathscr{I}_1,\mathscr{I}_2,\ldots, \mathscr{I}_k\}\\ |\mathscr{M}|=n}}(-1)^{k+1}\mathbb{P}_{p}(D_N^c,\text{ $\mathscr{M}$ occurs}).
\]
As the function $F_n$ is polynomial in $p$ on $[0,1]$, we can make sense of $F_n(z)$ for all complex $z$, even though the interpretation on the right-hand side ceases to be valid. For every multi-interface $\mathscr{M}$, we can partition the event $\{D_N^c,\text{ $\mathscr{M}$ occurs}\}$ into finitely many events $\mathcal{E}_{\mathscr{M},i}$ of the form $\{\text{all edges in $O_{\mathscr{M},i}$ are open and all edges in $C_{\mathscr{M},i}$ are closed}\}$ with $O_{\mathscr{M},i}$ and $C_{\mathscr{M},i}$ edge-sets of cardinality at most $An$.  We set $G_{\mathscr{M},i}(z):=z^{|O_{\mathscr{M},i}|}(1-z)^{|C_{\mathscr{M},i}|}$ and observe that, for $p\in[0,1]$, we have $G_{\mathscr{M},i}(p)=\mathbb{P}_p(\mathcal{E}_{\mathscr{M},i})$.

By the triangle inequality, for every complex number $z$, we have
\[
|F_n(z)|\le\sum_{\mathscr{M}\,:\,|\mathscr{M}|=n}\sum_i |G_{\mathscr{M},i}(z)|.
\]
Letting $r<p_0$ and $a(r):=A\log(\frac{p_0+r}{p_0-r})$, we can now apply \cite[Corollary 4.3]{analyticity} to get that for every $z$ satisfying $|z-p_0|\le r$, for all $\mathscr{M}$ and $i$, we have $|G_{\mathscr{M},i}(z)|\le e^{a(r)n}\mathbb{P}_{p_0-r}(\mathcal{E}_{\mathscr{M},i})$. Therefore, for every complex number $z$ satisfying $|z-p_0|\le r$, we have
\[
|F_n(z)|\le e^{a(r)n}\sum_{\mathscr{M}\,:\,|\mathscr{M}|=n}\sum_i \mathbb{P}_{p_0-r}(\mathcal{E}_{\mathscr{M},i})\,=\,e^{a(r)n} \sum_{\substack{\mathscr{M}~:~ |\mathscr{M}|=n}}\mathbb{P}_{p_0-r}(D_N^c,\text{ $\mathscr{M}$ occurs}).
\]

Observe that $a(r)$ converges to $0$ as $r$ tends to $0$. By Proposition~\ref{prop: exp dec}, we have
\[
\sum_{\substack{\mathscr{M}~:~ |\mathscr{M}|=n}}\mathbb{P}_{p_0-r}(D_N^c,\text{ $\mathscr{M}$ occurs})\le C e^{-cn},
\]
provided that $r<\varepsilon$. We can find $r>0$ small enough so that $a(r)\leq c/2$, thus obtaining that 
\begin{equation}\label{eq:exp dec}
\left\lvert F_n(z)\right\rvert\leq Ce^{-cn/2}    
\end{equation}
for every $|z-p_0|\leq r$. Therefore, the inclusion-exclusion expansion $\sum_{n\geq 1}F_n(z)$ converges uniformly in the disk $\{z\in \mathbb{C}~:~ |z-p_0|\leq r\}$, which implies that it is analytic in its interior, so that $\theta$ admits an analytic extension to this set. 
\end{proof}

We now proceed to the proof of Theorem~\ref{thm:analyticity-local}.

\begin{proof}[Proof of Theorem~\ref{thm:analyticity-local}]
Let $\mathscr{G}\in \mathfrak{G}$ and consider $p_0>p_c(\mathscr{G})$ and a large enough constant $k_0>0$. Consider $\mathscr{H}\in \mathfrak{G}$ such that $R(\mathscr{G},\mathscr{H})\geq k_0$. The value of $k_0$ will be chosen in such a way that $\theta_{\mathscr{H}}$ is analytic in $[p_0,1]$.

By Theorem~\ref{thm:finitary-timar}, $k_0$ can be chosen to be large enough so that $\weak(\mathscr{H})\le t$ for some constant $t>0$ depending only on $\mathscr{G}$. Having fixed the value of $t$, choose $N$, $\varepsilon$, and $r$ as in the proof of Theorem~\ref{thm:analyticity-poly} (the value of these constants depends in general on the value of $t$, but in light of Proposition~\ref{prop: exp dec} it does not depend on $k_0$). By increasing the value of $k_0$ if necessary, we can further assume that $k_0\geq 2t+10(t+1)N$. Then the formula
\begin{equation}\label{eq:theta_h expansion}
1-\theta_{\mathscr{H}}(p)=\mathbb{P}_{p,\mathscr{H}}(D_N)+\mathbb{P}_{p,\mathscr{H}}(D_N^c,\text{ an interface occurs})
\end{equation}
holds for every $p\in (p_0-\varepsilon,p_0+\varepsilon)\cap [0,1]$. Furthermore, $\mathbb{P}_{p,\mathscr{H}}(D_N)$ is a polynomial in $p$, hence analytic. To deduce that $\theta_{\mathscr{H}}$ is analytic on a neighbourhood of $p_0$, we use the inclusion-exclusion principle as in the proof of Theorem~\ref{thm:analyticity-poly} to expand $\mathbb{P}_{p,\mathscr{H}}(D_N^c,\text{ an interface occurs})$ as a series of polynomials, and then note that with our choice of constants, inequality \eqref{eq:exp dec} holds, hence the series converges uniformly in the disk $\{z\in \mathbb{C}~:~ |z-p_0|\leq r\}$ and admits an analytic extension there. By compactness of the interval $[p_0,1]$, we can find a complex neighbourhood $U$ of $[p_0,1]$ where both $\theta_{\mathscr{H}}$ and $\theta_{\mathscr{G}}$ admit analytic extensions $f_{\mathscr{H}}$ and $f_{\mathscr{G}}$, respectively, and $U$ can be chosen uniformly over all $\mathscr{H}$ as above. 

Now we proceed to show that in a complex neighbourhood of $p_0$, the quantity $\|f_{\mathscr{G}}-f_{\mathscr{H}}\|_\infty$ decays exponentially in $R(\mathscr{G},\mathscr{H})$. Consider $k\geq k_0$ and some $\mathscr{H}\in \mathfrak{G}$ such that $R(\mathscr{G},\mathscr{H})\geq k$. Let $m\geq 1$ be such that $k\geq 2tm+20(t+1)N$. To derive an estimate on $f_{\mathscr{G}}-f_{\mathscr{H}}$, we use \eqref{eq:theta_h expansion} and the fact that $\mathbb{P}_{p,\mathscr{H}}(D_N)=\mathbb{P}_{p,\mathscr{G}}(D_N)$ by our choice of $k_0$ to obtain
$$\theta_{\mathscr{G}}(p)-\theta_{\mathscr{H}}(p)=\mathbb{P}_{p,\mathscr{H}}(D_N^c,\text{ an interface occurs})-\mathbb{P}_{p,\mathscr{G}}(D_N^c,\text{ an interface occurs}),$$ which holds for every $p\in (p_0-\varepsilon,p_0+\varepsilon)\cap [0,1]$. Thus, it remains to analyse the latter difference.

To this end, we again expand both probabilities using the inclusion-exclusion principle. Consider a multi-interface $\mathscr{M}=\{\mathscr{I}_1,\ldots,\mathscr{I}_k\}$ of size $n\leq m$ in $\mathscr{H}$. Note that $\mathscr{M}$ is contained in the ball of radius $tn-1+t(n-1)\leq 2tn-1$ by Lemma~\ref{lem:geodesic} and $t$-connectivity. Hence the event $\{D_N^c,\text{ $\mathscr{M}$ occurs}\}$ depends only on the state of the edges in the ball of radius $2tn+20(t+1)N$ centred at $o$. 
By our assumption on $R(\mathscr{G},\mathscr{H})$,
$$\sum_{\substack{\mathscr{M}: |\mathscr{M}|\leq m}}(-1)^{i+1}\mathbb{P}_{p,\mathscr{H}}(D_N^c,\text{ $\mathscr{M}$ occurs})=\sum_{\mathscr{M}: |\mathscr{M}|\leq m}(-1)^{i+1}\mathbb{P}_{p,\mathscr{G}}(D_N^c,\text{ $\mathscr{M}$ occurs}),$$
where here we abuse notation and denote both $\mathscr{M}$ and its isomorphic image in $\mathscr{G}$ by the same symbol. 
It follows now from \eqref{eq:exp dec} and the triangle inequality that 
$$|f_{\mathscr{G}}(z)-f_{\mathscr{G}}(z)|\leq 2C\sum_{i>m}e^{-ci/2}$$
for every $|z-p_0|\leq r$, with $r>0$ is chosen as in \eqref{eq:exp dec}. This concludes the proof.
\end{proof}

\section{Strong probabilistic estimates in the supercritical regime}
\label{sec:strong-proba-estimates}

In this section, we prove Theorem~\ref{thm:sharpness-local} and its analogue for the cluster size.

\hypertarget{target:sharpness}{
\subsection{Supercritical sharpness}}

Our aim is to prove the local version of supercritical sharpness stated in Theorem~\ref{thm:sharpness-local}. Let us point out that subcritical sharpness, namely the main result of \cite{duminil2016new}, can easily be made local for all transitive graphs. Indeed, \cite{duminil2016new} characterises subcriticality via the finite-size criterion ``$\varphi_p(S)<1$''. 

\begin{proof}[Proof of Theorem~\ref{thm:sharpness-local}]
Let $p_0\in (p_c(\mathscr{G}),1]$. Let us first show that $p_c(\mathscr{H})<p_0$ for every $\mathscr{H}\in \mathfrak{G}$ in a neighbourhood of $\mathscr{G}$. To this end, we apply Theorem~\ref{thm:analyticity-local} to get that for some $r>0$ and $\varepsilon>0$ we have $\theta_{\mathscr{H}}(p)\geq \varepsilon$ for every $\mathscr{H}\in \mathfrak{G}$ with $R(\mathscr{H},\mathscr{G})\geq r$ and every $p\in (p_0-\varepsilon,p_0+\varepsilon)\cap [0,1]$. It follows that $p_c(\mathscr{H})<p_0$. Alternatively, we could use \cite{cmtlocality}.

Let us now prove the second part of the statement. It suffices to establish the desired bound for $p$ in an open neighbourhood of $p_0$, as then one can use compactness of the interval $[p_0,1]$ to conclude. Recall Proposition~\ref{prop: exp dec}, and note that it implies the following bound: there exist $N,r,c',\varepsilon>0$ such that for every $\mathscr{H}\in \mathfrak{G}$ with $R(\mathscr{H},\mathscr{G})\geq r$ we have
\begin{equation}\label{eq:bad exp decay}
\mathbb{P}_{p,\mathscr{H}}(|\partial \mathcal{C}_x^{\mathrm{bad}}(N)|\geq n)\leq e^{-c'n}    
\end{equation}
for every $p\in (p_0-\varepsilon,p_0+\varepsilon)\cap[0,1]$ and every $n\ge1$. Here we used that $\mathbb{P}_{p,\mathscr{H}}(|\partial \mathcal{C}_x^{\mathrm{bad}}(N)|\geq n)<1$ and the measurability of the event $\{|\partial\mathcal{C}_x^{\mathrm{bad}}(N)|\geq n\}$ with respect to a ball of bounded radius to remove the prefactor $C>0$ of Proposition~\ref{prop: exp dec}. Now notice that on the event $\{u\longleftrightarrow v \centernot\longleftrightarrow \infty\}$ we have $20N<\mathrm{diam}(\mathcal{C}_u)<\infty$, as long as $d_\mathscr{H}(u,v)>20N$. This in turn implies that $\partial\mathcal{C}_u(N)$ is a finite minimal cutset that disconnects $\mathcal{C}_u$ from infinity. Applying \cite[Lemma 2.3]{supercritpoly} we obtain that $|\partial\mathcal{C}_u(N)|\geq \frac{\mathrm{diam}(\mathcal{C}_u)}{2}$, hence $|\partial \mathcal{C}_u(N)|\geq \frac{d_\mathscr{H}(u,v)}{2}$. Thus, we can apply \eqref{eq:bad exp decay}, Proposition~\ref{prop: bad vert}, Lemma~\ref{lem:geodesic} and a union bound to deduce that
$$\mathbb{P}_{p,\mathscr{H}}(u\longleftrightarrow v \centernot\longleftrightarrow \infty) \le \sum_{n\geq \frac{d_\mathscr{H}(u,v)}{2}}tne^{-c'n}$$
for every $u,v$ that satisfy $d_\mathscr{H}(u,v)>20N$. Thus, $\mathbb{P}_{p,\mathscr{H}}(u\longleftrightarrow v \centernot\longleftrightarrow \infty) \le \exp(-c' d_\mathscr{H}(u,v)/4)$, whenever $d_\mathscr{H}(u,v)$ is large enough. To handle small distances, we use that the event $\{u\longleftrightarrow v \centernot\longleftrightarrow \infty\}$ has probability bounded away from $1$, uniformly over $\mathscr{H}$ in a neighbourhood of $\mathscr{G}$, since we can disconnect $u$ and $v$ by simply closing all edges incident to $u$. This completes the proof.
\end{proof}

\begin{remark}
    Getting an exponential lower bound is easy and well known. Given $\varepsilon>0$ and a transitive graph $\mathscr{H}$, there is a positive constant $c''$ depending only on $\varepsilon$ and the degree of $\mathscr{H}$ such that \[\forall p\in[\varepsilon,1-\varepsilon],\qquad\mathbb{P}_{p,\mathscr{H}}(u\longleftrightarrow v \centernot \longleftrightarrow\infty)\geq e^{-c''d_\mathscr{H}(u,v)}.\] Notice that the degree of $\mathscr{H}$ equals that of $\mathscr{G}$ as soon as we assume $R(\mathscr{G},\mathscr{H})\ge1$. To obtain such a bound, it suffices to open a geodesic path $\gamma$ of length $d_\mathscr{H}(u,v)$ from $u$ to $v$, and close all other edges with at least one endpoint in $\gamma$. 
\end{remark}

\subsection{Probability that a cluster is large and finite}
\label{sec:large-cluster}

In this section, our aim is to handle the decay of the probability that the cluster of the origin is large and finite in the supercritical regime. In the case of the hypercubic lattice $\mathbb{Z}^d$, it is classical that the probability of the latter event decays stretched exponentially with exponent equal to $\frac{d-1}{d}$; see \cite{kesten1990probability}. This result can be extended to all $d$-dimensional transitive graphs, where now the dimension $d(\mathscr{G})$ of a graph $\mathscr{G}$ of polynomial growth is defined to be the smallest exponent $d$ such that $|B_n|\leq Cn^d$.

In general, two graphs of polynomial growth might not have the same dimension, regardless of how close they are. For this reason, probabilities of the form $\mathbb{P}_{p,\mathscr{G}}(k<|\mathcal{C}_o|<\infty)$ do not have the same order of stretched exponential decay as one varies the underlying graph $\mathscr{G}$. Nevertheless, we will show that the rate $c$ of stretched exponential decay can be chosen to be uniform.

\begin{theorem}\label{thm:cluster size}
Let $\mathscr{G}\in\mathfrak{G}$ and $p_0\in (p_c(\mathscr{G}),1]$. Then, there is a constant $c>0$ such that the set of all $\mathscr{H}\in\mathfrak{G}$ satisfying $p_c(\mathscr{H})<p_0$ and the following condition for all $k$ is a neighbourhood of the graph $\mathscr{G}$:
    \[ \forall p\in [p_0,1], \quad
\mathbb{P}_{p,\mathscr{H}}(k<|\mathcal{C}_o|<\infty)\le \exp\left(-ck^{\frac{d(\mathscr{H})-1}{d(\mathscr{H})}}\right).
    \]
\end{theorem}

In order to prove the above theorem, we need the following result. We expect this result to be known to the experts. We include a proof, citing relevant works, for convenience of the reader. Recall that $\partial_V A$ denotes the set of vertices that do not belong to $A$ but are adjacent to an element of $A$.

\begin{proposition}\label{prop:isoperimetric inequality}
Let $d\geq 1$. There exists $\varepsilon=\varepsilon(d)>0$ such that for every $d$-dimensional transitive graph, we have $|\partial_V A|\geq \varepsilon |A|^{\frac{d-1}{d}}$ for every finite set of vertices $A$.
\end{proposition}

\begin{proof}
We apply \cite[Corollary 1.5]{tesseratointon} for dimension $d-1$, which provides a uniform constant $C=C(d-1)>0$ such that every $d$-dimensional transitive graph $\mathscr{G}$ satisfies $|B_r|> r^{d}/C$ for every $r\geq C$, because otherwise $\mathscr{G}$ would be of dimension at most $d-1$. By increasing the value of $C$, we can make this conclusion to hold for all $r\ge 1$. Then the desired result follows from \cite[Proposition 5.1]{tessera2020sharp}.
\end{proof}

We now prove Theorem~\ref{thm:cluster size}.

\begin{proof}[Proof of Theorem~\ref{thm:cluster size}]
Let $\mathscr{G}\in \mathfrak{G}$ and recall \eqref{eq:bad exp decay}. Arguing as in the proof of Theorem~\ref{thm:sharpness-local}, we see that it suffices to prove that there exists $\varepsilon>0$ such that for all $\mathscr{H}\in \mathfrak{G}$ lying in some neighbourhood of $\mathscr{G}$ we have $|\partial\mathcal{C}_o(N)|\geq \varepsilon |\mathcal{C}_o|^{\frac{d(\mathscr{H})-1}{d(\mathscr{H})}}$ whenever $20N<\mathrm{diam}(\mathcal{C}_o)<\infty$. By \cite[Corollary 1.4]{tesseratointon}, there exists $r\ge 1$ such that for every $\mathscr{H}\in \mathfrak{G}$ with $R(\mathscr{H},\mathscr{G})\geq r$, the dimension $d(\mathscr{H})$ is uniformly bounded. For such a $\mathscr{H}$, applying Proposition~\ref{prop:isoperimetric inequality} to the complement of the infinite connected component of $\mathscr{H}\setminus \mathcal{C}_o(N)$ yields the desired inequality for the exterior exposed boundary\footnote{The exterior exposed boundary of a finite set $A$ is the set of vertices that belong to an infinite connected component of the complement of $A$ and that furthermore are adjacent to some vertex in $A$.} of $\mathcal{C}_o(N)$.
Besides, as $r\ge 1$, these $\mathscr{H}$ all have the same degree, so that moving from exterior exposed boundary to exposed boundary can only distort inequalities by a bounded factor.
\end{proof}

\begin{remark}
    We now sketch how one can obtain a matching lower bound at the stretched exponential scale, namely that for every $p_0>p_c(\mathscr{G})$ and $\varepsilon>0$ there exists $c'>0$ such that the set of all $\mathscr{H}\in\mathfrak{G}$ satisfying the following condition for all $k$ is a neighbourhood of  $\mathscr{G}$:
    \[\forall p\in [p_0,1-\varepsilon], \quad
\mathbb{P}_{p,\mathscr{H}}(k<|\mathcal{C}_o|<\infty)\ge \exp\left(-c'k^{\frac{d(\mathscr{H})-1}{d(\mathscr{H})}}\right).
    \] 
 Let $\mathscr{G}\in \mathfrak{G}$ and $p_0\in(p_c(\mathscr{G}),1]$. By Theorem~\ref{thm:analyticity-local}, there exists $\varepsilon_0>0$ such that $\theta_{\mathscr{H}}(p)\geq \varepsilon_0$ for every $\mathscr{H}$ close enough to $\mathscr{G}$ and every $p\geq p_0$. Now consider $k\geq 1$, and let $\mathcal{N}$ be the number of vertices $x\in B_k$ such that the event $\{x\longleftrightarrow \partial B_k\}$ happens. Then, $\mathbb{E}_{p}(\mathcal{N})\geq \varepsilon_0 |B_k|$. Letting $\delta:=\mathbb{P}_p(\mathcal{N}\geq \varepsilon_0 |B_k|/2)$, we see that 
    $$
    \varepsilon_0 |B_k|\leq \mathbb{E}_{p}(\mathcal{N})\leq (1-\delta)\varepsilon_0 |B_k|/2 +\delta|B_k|,
    $$
    which implies that $\delta\geq \varepsilon_0/2$. Since $\partial B_k$ is $t$-coarse connected, we can now open all edges in the $t$-neighbourhood of $\partial B_k$ to connect all clusters intersecting $\partial B_k$. By taking the intersection with the event $\{o\longleftrightarrow \partial B_k\}$ and applying the Harris--FKG inequality we obtain that \[\mathbb{P}_{p}(|\mathcal{C}_o\cap B_{k+t}|\geq \varepsilon |B_k|/2)\geq e^{-c'|\partial B_k|}.\] If it is indeed the case that $|\partial B_k|/|B_k|=O(1/k)$, then we can conclude by closing all edges of the infinite connected component of $\mathscr{H}\setminus B_k$ that intersect both $S_{k+t}$ and $S_{k+t+1}$. Lemma~4.2 from \cite{supercritpoly} guarantees that this occurs at sufficiently many scales $k$ for the argument to adapt.
\end{remark}

\subsection{Exponential cluster repulsion and isoperimetry}\label{sec:cluster repulsion}

In this section, our aim is to prove that the probability $\mathcal{C}_o$ is finite but touches the infinite cluster at $t$ places decays exponentially in $t$. As a corollary, we obtain an isoperimetric inequality for the infinite cluster.

Let us first give the formal definition. Consider a graph $\mathscr{G}\in \mathfrak{G}$. For two distinct clusters $\mathcal{C}$ and $\mathcal{C}'$, a \emph{touching edge} is an edge of $\mathscr{G}$ with one endpoint in $\mathcal{C}$ and the other in $\mathcal{C}'$. The number of such edges is denoted $\tau(\mathcal{C},\mathcal{C}')$. For supercritical percolation on $\mathscr{G}$, the almost surely unique infinite cluster is denoted $\mathcal{C}_{\infty}$.

We now state and prove the following result.

\begin{theorem}\label{thm:cluster repulsion}
Let $\mathscr{G}\in\mathfrak{G}$ and $p_0\in (p_c(\mathscr{G}),1]$. Then, there is a constant $c>0$ such that the set of all $\mathscr{H}\in\mathfrak{G}$ satisfying $p_c(\mathscr{H})<p_0$ and the following condition for all $k$ is a neighbourhood of the graph $\mathscr{G}$:
    \[ \forall p\in [p_0,1], \quad
\mathbb{P}_{p,\mathscr{H}}\big(|\mathcal{C}_o|<\infty,~ \tau(\mathcal{C}_o,\mathcal{C}_{\infty})\ge k\big)\le \exp(-ck).
    \]    
\end{theorem}
\begin{proof}
Let $p_0\in (p_c(\mathscr{G}),1]$. Arguing as in the proof of Theorem~\ref{thm:sharpness-local}, we see that it suffices to find an $\varepsilon>0$ such that the desired uniform exponential decay holds for every $p\in (p_0-\varepsilon,p_0+\varepsilon)\cap [0,1]$. 

To this end, consider some large enough constants $N,r>0$ and a graph $\mathscr{H}\in \mathfrak{G}$ such that $R(\mathscr{G},\mathscr{H})\ge r$. Consider a configuration $\omega$ such that $\mathcal{C}_o$ is finite, and note that $|\mathcal{C}_o|\ge \tau(\mathcal{C}_o,\mathcal{C}_{\infty})/D$, where $D$ is degree of $\mathscr{G}$. Hence, if $\tau(\mathcal{C}_o,\mathcal{C}_{\infty})>D|B_{20N,\mathscr{G}}|$ and $r\ge 20N$, then $\textup{diam}(\mathcal{C}_o)>20N$. Consider the associated interface $\mathcal{I}$ and recall that the closed edges of $\overline{\mathcal{I}}$ form a bond-cutset between $o$ and infinity by Lemma~\ref{lem:cutset}. Let us denote this bond-cutset $\Pi$. Then $\Pi$ needs to contain all touching edges between $\mathcal{C}_o$ and $\mathcal{C}_{\infty}$. Hence we have 
\[
|\mathcal{I}|\ge \frac{|\overline{\mathcal{I}}|}{|B_{10N,\mathscr{G}}|}\ge \frac{\tau(\mathcal{C}_o,\mathcal{C}_{\infty})}{D|B_{10N,\mathscr{G}}|}.
\]
The desired result follows from Proposition~\ref{prop: exp dec} for large values of $k$.
Finally, since we have $|\mathcal{C}_o|\ge \tau(\mathcal{C}_o,\mathcal{C}_{\infty})/D$, we can use Theorem~\ref{thm:cluster size} to handle the small values of $k$. 
\end{proof}

Given a finite subset $S$ of $\mathcal{C}_{\infty}$, we distinguish between two notions of edge boundary in~$\mathcal{C}_{\infty}$. Following Pete's notation \cite{pete2008note}, we denote $\partial^+_{\mathcal{C}_{\infty}} S$
the set of open edges of $\mathcal{C}_{\infty}$ with one endpoint in $S$ and the other in the infinite connected component of $\mathcal{C}_{\infty}\setminus S$. We also denote $\tilde\partial_{\mathcal{C}_{\infty}} S$
the set of edges in $\mathscr{G}$ with one endpoint in $S$ and the other in the infinite connected component of $\mathcal{C}_{\infty}\setminus S$.

We now prove the following anchored isoperimetric inequality for the infinite cluster. The proof is similar to that of \cite[Theorem 1.2]{pete2008note}.

\begin{theorem}
Let $\mathscr{G}\in\mathfrak{G}$ and $p_0\in (p_c(\mathscr{G}),1]$. Then, there is a constant $c>0$ such that the set of all $\mathscr{H}\in\mathfrak{G}$ satisfying $p_c(\mathscr{H})<p_0$ and the following condition is a neighbourhood of the graph $\mathscr{G}$: for every $p\in [p_0,1]$, there exists $\alpha>0$ such that for all $M\geq 1$, we have
\[\mathbb{P}_p\left(\exists S\text{ connected}: o\in S\subset \mathcal{C}_{\infty},\ M\leq |S|<\infty,\  \frac{|\partial^+_{\mathcal{C}_{\infty}} S|}{|S|^{1-1/d(\mathscr{H})}}\leq \alpha\right)\leq \exp\left(-cM^{1-1/d(\mathscr{H})}\right).
\]
\end{theorem}
\begin{remark}
We note that the value of $\alpha$ that the argument gives depends on $p$ and degenerates to $0$ as $p$ tends to $1$, while the value of $c$ depends only on $p_0$.    
\end{remark}

\begin{proof}
The result holds for $p=1$ by Proposition~\ref{prop:isoperimetric inequality}, so let us assume that $p_0\in (p_c(\mathscr{G}),1)$ and prove the result for $p\in [p_0,1)$. One can handle the case where $|S|=m$ for some small $m$ by opening all edges in the ball of radius $2m$ and forcing the ball to connect to infinity to ensure that $\frac{|\partial^+_{\mathcal{C}_{\infty}} S|}{|S|^{1-1/d(\mathscr{H})}}$ remains bounded away from $0$, so it remains to handle the case where $M$ is large enough. 

Define the event
\[
\mathcal{Y}(m,t):=\left\{|\mathcal{C}_o|=m \quad\text{and}\quad \tau(\mathcal{C}_o,\mathcal{C}_\infty)=t\right\}.
\]
By Theorems~\ref{thm:cluster size} and \ref{thm:cluster repulsion}, there exist $r,c>0$ such that for every graph $\mathscr{H}\in \mathfrak{G}$ satisfying $R(\mathscr{G},\mathscr{H})\ge r$ and every $p\geq p_0$, we have
\[
\mathbb{P}_p(\mathcal{Y}(m,t))\leq \exp\bigl(-c\max\{m^{1-1/d},t\}\bigr),
\]
where $d$ denotes the dimension of $\mathscr{H}$. We now introduce the events\[
\mathcal{X}(m,t,s):=\left\{\exists S\text{ connected}: o\in S\subset \mathcal{C}_{\infty},\ |S|=m,\ |\tilde\partial_{\mathcal{C}_{\infty}} S|=t,\ |\partial^+_{\mathcal{C}_{\infty}} S|=s\right\}.
\]    
and 
\[
\mathcal{E}_{\alpha,M}:=\left\{\exists S\text{ connected}: o\in S\subset \mathcal{C}_{\infty},\ M\leq |S|<\infty,\  \frac{|\partial^+_{\mathcal{C}_{\infty}} S|}{|S|^{1-1/d}}\leq \alpha\right\}.
\]
By a union bound, and ignoring integer parts to simplify the notation, we have
\begin{equation}\label{eq:bound 1}
\mathbb{P}_p\left(\mathcal{E}_{\alpha,M}\right)\leq \sum_{m\geq M}\sum_{t=1}^{Dm}\sum_{s=1}^{c_{m,t}}\mathbb{P}_p(\mathcal{X}(m,t,s)),
\end{equation}
where $c_{m,t}=\min\{\alpha m^{1-1/d},t\}$, and $d,D$ are the dimension and degree of $\mathscr{H}$, respectively. We now bound $\mathbb{P}_p(\mathcal{X}(m,t,s))$ using the multi-valued map principle.

Given a configuration $\xi\in \{0,1\}^{E(B_{2m})}$, we let $\mathcal{X}(m,t,s,\xi)$ denote the event $\{\omega|_{E(B_{2m})}=\xi \text{ and } \omega\in\mathcal{X}(m,t,s)\}$. For a $\xi$ such that $\mathbb{P}_p(\mathcal{X}(m,t,s,\xi))>0$ and a finite connected set $S\ni o$, define a new event $F(\xi,S)\subset\mathcal{Y}(m,t)$ by closing the edges in $\partial^+_{\mathcal{C}_{\infty}} S$. Note that $\mathbb{P}(F(\xi,S))=Q^s\mathbb{P}(\mathcal{X}(m,t,s,\xi))$, where $Q=\frac{p}{1-p}$. For each $\omega'$ in the image of $F$ such that $\omega'\in\mathcal{Y}(m,t)$, there are ${t\choose s}$ pre-images $(\xi,S)$ under $F$. Hence, we have
\[
\mathbb{P}_p(\mathcal{X}(m,t,s))\leq {t\choose s}Q^s\mathbb{P}_p(\mathcal{Y}(m,t)).
\]
With these bounds at hand, we estimate $\sum_{t=1}^{Dm}\sum_{s=1}^{c_{m,t}}\mathbb{P}_p(\mathcal{X}(m,t,s))$ by splitting it into $2$ sums according to whether $t\leq K\alpha m^{1-1/d}$ or $t>K\alpha m^{1-1/d}$, where $K=Q+2+4/c$.
To this end, note that for $\alpha>0$ small enough to satisfy $K \alpha \log(1+Q)<c/2$, we have that for every $m$ large enough,
\begin{equation}\label{eq:bound 2}
\sum_{t=1}^{K \alpha m^{1-1/d}} \; \sum_{s=1}^{c_{m,t}}\ \mathbb{P}_p(\mathcal{X}(m,t,s))\leq \sum_{t=1}^{K \alpha m^{1-1/d}} (1+Q)^t\exp\bigl(-cm^{1-1/d}\bigr)\leq \exp\left(-\frac{c}{2}m^{1-1/d}\right).
\end{equation}

For the second sum, note that for $t>K\alpha m^{1-1/d}>2\alpha m^{1-1/d}$ and $s\leq \alpha m^{1-1/d}$, we have
\[
{t \choose s}Q^s\leq {t \choose \alpha  m^{1-1/d}}(1+Q)^s\leq \left(\frac{et}{\alpha  m^{1-1/d}}\right)^{\alpha  m^{1-1/d}}(1+Q)^{\alpha  m^{1-1/d}},
\]
hence
\[
\sum_{s=1}^{\alpha m^{1-1/d}} {t \choose s} Q^s \leq \alpha m^{1-1/d}\exp\Bigl(\alpha\bigl(1+\log (\beta/\alpha)+\log(1+Q)\bigr)\,m^{1-1/d}\Bigr),
\]
where $\beta=t/m^{1-1/d}$. Since $\beta/\alpha> 1+Q$ and $\alpha<cK\alpha/4<c\beta/4$, if $\alpha$ is sufficiently small, then we have
\[
\alpha(1+\log(\beta/\alpha)+\log(1+Q)) < \alpha(1+2\log(\beta/\alpha))  < c\beta/2.
\]
Hence, for every $m$ large enough, we get
\begin{equation}\label{eq:bound 3}
\sum_{t=1+Q \alpha m^{1-1/d}}^{Dm} \; \sum_{s=1}^{c_{m,t}}\ \mathbb{P}_p(\mathcal{X}(m,t,s))\leq \sum_{t>K \alpha m^{1-1/d}} \exp\left(-\frac{c}{2}t\right).
\end{equation}
The desired result follows from putting together \eqref{eq:bound 1}, \eqref{eq:bound 2} and \eqref{eq:bound 3}.
\end{proof}

\begin{remark}
The above theorem implies in particular the following result: for every $\mathscr{G}\in \mathfrak{G}$, the anchored isoperimetric dimension of the cluster of the origin is equal to the dimension of $\mathscr{G}$ as soon as this cluster is infinite. This result for fixed graphs can alternatively be obtained from arguments in \cite{pete2008note} and \cite{supercritpoly}. The parallel work \cite{AK26} provides a more general result covering long-range percolation for fixed graphs in $\mathfrak{G}$ for which supercritical sharpness holds; see their Theorem~1.5. Still, we believe that the viewpoint of interfaces adopted in the present paper is a relevant way to revisit the techniques from \cite{pete2008note}, even for the hypercubic lattice $\mathbb{Z}^d$.
\end{remark}

\hypertarget{target:equiv}{\section{Context and proof of Proposition~\ref{prop:equiv}}}
\label{sec:equiv}

One of the biggest open problems in percolation theory is whether there is an infinite cluster at criticality for Bernoulli percolation on $\mathbb{Z}^d$ ($d\ge 2$) with the usual nearest neighbour structure. This question has been answered in the negative in dimension $2$ \cite{KestenCritical} and in dimensions $d\geq 11$ \cite{HS90, FH17}, but remains open in the intermediate dimensions. Recently, this question has attracted interest beyond the setting of the hypercubic lattice.

Recall that we define $\Ab$ to be the set of all graphs that can be written as a product of finitely many cycle-graphs with at least two infinite factors; see page~\pageref{page:ab}. Let us call such a graph a \defini{slab} if the number of infinite factors is exactly 2.
In \cite{slabs}, it is proved that slabs satisfy $\theta(p_c)=0$. See also \cite{BLPS99, Timar06, PPS06, H16, HH21} for related results on graphs growing faster than $\mathbb{Z}^d$.

One of the main motivations for studying the critical behaviour of percolation on slabs is the fact that for $d\geq 3$, the lattice $\mathbb{Z}^d$ can be obtained as the local limit of $\mathbb{Z}^2\times (\mathbb{Z}/n\mathbb{Z})^{d-2}$, potentially offering a pathway to understanding the critical behaviour of percolation on $\mathbb{Z}^d$.

This context being given, we now prove Proposition~\ref{prop:equiv} which, in the framework of $\Ab$, relates several regularity properties of the $\theta$ map to the $\theta(p_c)=0$ conjecture.

\begin{proof}[Proof of Proposition~\ref{prop:equiv}]
We shall prove the following implications, which generate all others:
\vspace{0.4cm}

~\hfill\begin{tikzcd}
  \ref{item:1} \arrow[rr, Leftrightarrow] && \ref{item:2}\arrow[rr, Leftrightarrow, blue]\arrow[d, Rightarrow,blue]&&\ref{item:5} \\
  \ref{item:3} \arrow[u, Rightarrow,blue] && \ref{item:4}\arrow[ll, Rightarrow]&&
\end{tikzcd}\hfill~

\vspace{0.4cm}

\noindent Only the blue implications remain to be proved, since $\ref{item:4}\implies\ref{item:3}$ is trivial and $\ref{item:1}\iff\ref{item:2}$ is well known; the latter follows, for example, from our Theorem~\ref{thm:analyticity-poly} and the right-continuity of $\theta$. There is no need to take care of \ref{item:6} as it is by definition equivalent to the conjunction $\ref{item:2}\&\ref{item:5}$.

Let $\mathscr{G}\in \Ab$. If $\mathscr{G}$ is the product of the square lattice with finitely many finite cycles (possibly none), then it follows from \cite{slabs} that $\theta_\mathscr{G}$ is continuous. Otherwise, $\mathscr{G}$ can be written as $\mathbb{Z}^2\times\mathbb{Z}^k\times \mathscr{H}$ with $k\ge 1$ and $\mathscr{H}$ a finite product of finite cycles. In this case, we can introduce the graph sequence $\mathscr{G}_n$ corresponding to $\mathbb{Z}^2\times(\mathbb{Z}/n\mathbb{Z})^k\times \mathscr{H}$, which indeed satisfies $\mathscr{G}_n\xrightarrow[n\to\infty]{}\mathscr{G}$.

Let us use this construction to prove $\ref{item:5}\implies \ref{item:2}$. By \cite{slabs}, each $\theta_{\mathscr{G}_n}$ is continuous. Therefore, if \ref{item:5} holds, the function $\theta_{\mathscr{G}}$ is continuous as the uniform limit of the continuous functions $\theta_{\mathscr{G}_n}$, whence $\ref{item:5}\implies \ref{item:2}$.
We can also use this construction to establish $\ref{item:3}\implies \ref{item:1}$. By \cite{MS19}, for every $n$, we have $p_c(\mathscr{G}_n)>p_c(\mathscr{G})$. In particular, we have $\theta_{\mathscr{G}_n}(p_c(\mathscr{G}))=0$. As a result, if \ref{item:3} holds, then $\theta_{\mathscr{G}}(p_c(\mathscr{G}))=\lim_{n}\theta_{\mathscr{G}_n}(p_c(\mathscr{G}))=0$.

At last, let us assume \ref{item:2} and prove that both \ref{item:4} and \ref{item:5} hold. Let $(\mathscr{G}_n)$ be a sequence of elements of $\Ab$ converging to some $\mathscr{G}$ in $\Ab$.
We know that:
\begin{itemize}
    \item by \ref{item:2}, the function $\theta_\mathscr{G}$ is continuous on $[0,1]$,
    \item every $\theta_{\mathscr{G}_n}$ is non-decreasing on $[0,1]$,
    \item by Corollary~\ref{coro:kdiff} and by lower semicontinuity of $p_c$ on the space of transitive graphs, the sequence $(\theta_{\mathscr{G}_n})$ converges pointwise to $\theta_\mathscr{G}$ on $[0,1]\setminus \{p_c(\mathscr{G})\}$.
\end{itemize}
It is easy to check that these three properties entail that $\theta_{\mathscr{G}_n}$ converges uniformly to $\theta_\mathscr{G}$. Therefore, the implication $\ref{item:2}\implies \ref{item:5}$ holds and, as $\theta_{\mathscr{G}}$ is uniformly continuous, so does the implication $\ref{item:2}\implies \ref{item:4}$.
\end{proof}

\begin{remark}
    We state Proposition~\ref{prop:equiv} in the setup of $\Ab$ rather than in that of $\mathfrak{G}$ because we need every graph of our class to be well approximated by graphs of our class that further satisfy $\theta(p_c)=0$. More precisely, for every $\mathscr{G}$, we need a sequence of graphs that are quotients of $\mathscr{G}$, converge to $\mathscr{G}$, and satisfy $\theta(p_c)=0$. Working with $\mathfrak{G}$ would raise two issues:
    \begin{itemize}
        \item it is not clear at all how to approximate a graph in $\mathfrak{G}$ by two-dimensional transitive graphs, 
        \item the article \cite{slabs} does not deal with all two-dimensional transitive graphs but only such graphs having all symmetries of the square lattice.
    \end{itemize}
Our arguments actually hold if $\Ab$ is replaced with the class of all graphs of the form $\mathbb{Z}^d\times\mathscr{H}$, where $\mathbb{Z}^d$ is endowed with its usual graph structure, $d$ is at least 2 and $\mathscr{H}$ is a finite transitive graph.
\end{remark}

\section{Investigating $\theta$ near $p_c$}
\label{sec:theta-near-pc}

The unconditional results of this paper are stated away from $p_c$. It is worthwhile investigating what happens near $p_c$ and, to this end, it is useful to introduce a few definitions.

For every $1\le n\le \infty$, let $f_n$ be a function from some domain $D_n$ to $\mathbb{R}$. In this section, we say that $(f_n)_{n\in\mathbb{N}}$ converges \defini{uniformly} to $f_\infty$ if
\[
\sup_{x\in D_n\cap D_\infty} |f_n(x) -f_\infty(x)|\xrightarrow[n\to\infty]{}0.
\]
Given $\mathscr{G}\in\mathfrak{G}$, we denote by $\tilde\theta_\mathscr{G}$ the restriction of $\theta_\mathscr{G}$ to $(p_c(\mathscr{G}),1]$.

\begin{proposition}
    The following statements are equivalent:
    \begin{itemize}
        \item every graph in $\Ab$ satisfies $\theta(p_c)=0$,
        \item for every sequence $(\mathscr{G})_{n\le \infty}$ of elements of $\Ab$ satisfying $\mathscr{G}_n\xrightarrow[n\to\infty]{}\mathscr{G}_\infty$, we have uniform convergence of $\tilde\theta_{\mathscr{G}_n}$ to $\tilde\theta_{\mathscr{G}_\infty}$.
    \end{itemize}
\end{proposition}
\begin{proof}
    The first condition implies the second one because Proposition~\ref{prop:equiv} asserts that its conditions \ref{item:1} and \ref{item:5} are equivalent.
    We now prove the reverse implication. Let $\mathscr{G}\in \Ab$. Construct $\mathscr{G}_n$ as in the proof of Proposition~\ref{prop:equiv}. Observe that $\mathscr{G}_n$ is a quotient of $\mathscr{G}$, so that Proposition~1 from \cite{MR1423907} guarantees that $p_c(\mathscr{G}_n)\ge p_c(\mathscr{G}_\infty)$. This means that the domain of $\tilde\theta_{\mathscr{G}_\infty}$ contains that of $\tilde\theta_{\mathscr{G}_n}$. As \cite{slabs} guarantees that $\theta_{\mathscr{G}_n}$ is continuous, we have $\inf_{p>p_c(\mathscr{G}_n)}\tilde\theta_{\mathscr{G}_n}(x)=0$. If we assume uniform convergence of $\tilde\theta_{\mathscr{G}_n}$ to $\tilde\theta_{\mathscr{G}_\infty}$, this yields $\inf_{p>p_c(\mathscr{G})}\tilde\theta_{\mathscr{G}}(x)\le0$, hence that $\mathscr{G}$ satisfies $\theta(p_c)=0$.
\end{proof}

We conclude by asking the following question.

\begin{question}
\label{what-is-your-quest}
Is it true that for every sequence $(\mathscr{G}_n)_{n\le \infty}$ of elements of $\Ab$ such that $\mathscr{G}_n$ converges to $\mathscr{G}_\infty$, we have uniform convergence of the derivative $\tilde\theta'_{\mathscr{G}_n}$ to $\tilde\theta'_{\mathscr{G}_\infty}$?
\end{question}

We believe that the answer is negative. Indeed, endowed with their usual graph structures, we have $\mathbb{Z}^2\times(\mathbb{Z}/n\mathbb{Z})^{98}\xrightarrow[n\to\infty]{}\mathbb{Z}^{100}$. Because the number 100 is large enough, it is known that, for $\mathbb{Z}^{100}$, the function $\tilde\theta'$ is bounded; see \cite{BA91, HS90}. In contrast, using the main result of \cite{KZderivative} and the mean-value theorem, one can see that, for the square lattice, the function $\tilde\theta'$ is unbounded. The universality paradigm inclines us to think that, as $\mathbb{Z}^2\times(\mathbb{Z}/n\mathbb{Z})^{98}$ is at large scale two-dimensional, it will behave as the square lattice at or near its critical point. It thus seems likely that every $\mathbb{Z}^2\times(\mathbb{Z}/n\mathbb{Z})^{98}$ has unbounded $\tilde \theta'$. This would be incompatible with $\tilde\theta'$ being bounded for $\mathbb{Z}^{100}$ and uniform convergence of $\tilde\theta'$ holding.

It is interesting to notice that, for the notion of convergence studied in this section, one expects different behaviours for $\tilde\theta$ and its derivative, while Corollary~\ref{coro:kdiff} asserts that, away from $p_c$, everything behaves nicely for $\theta$, $\theta'$, and derivatives of higher order.

We conclude by observing that if the answer to Question~\ref{what-is-your-quest} were to be affirmative, this would imply $\theta(p_c)=0$ on $\Ab$, as $\theta(p_c)=1-\int_{p_c}^1\theta'(p)\,\mathrm{d}p$. Indeed, one can approximate $\mathscr{G}$ by slabs $\mathscr{G}_n$ as in the proof of Proposition~\ref{prop:equiv} and, were the answer affirmative, the equality with a star would be valid in the following computation:
\begin{eqnarray*}
\theta_{\mathscr{G}}(p_c(\mathscr{G}))&=&1-\int_{p_c(\mathscr{G})}^1\theta_{\mathscr{G}}'(p)\,\mathrm{d}p=1-\lim_n \int_{p_c(\mathscr{G}_n)}^1\theta_{\mathscr{G}}'(p)\,\mathrm{d}p\\&\overset{\star}{=}& 1-\lim_n \int_{p_c(\mathscr{G}_n)}^1\theta_{\mathscr{G}_n}'(p)\,\mathrm{d}p=\lim_n \theta_{\mathscr{G}_n}(p_c(\mathscr{G}_n))=0.
\end{eqnarray*}

\newpage

\small 
\bibliography{biblio}
\bibliographystyle{alpha}
\end{document}